\documentclass[11pt]{amsart} 
\usepackage{latexsym}\usepackage{amsmath}\usepackage{xspace}\usepackage{enumerate}\usepackage{amsfonts}\usepackage{amscd}\usepackage{syntonly}\usepackage{amssymb}

\usepackage{hyperref}

\usepackage{bbm}

\usepackage{amssymb}
\usepackage{mathrsfs}

\usepackage{mathtools}
 
\usepackage{latexsym}  
 
\usepackage[alwaysadjust]{paralist}

\newtheorem{thm}{Theorem}
\newtheorem{lem}[thm]{Lemma}
\newtheorem{prop}[thm]{Proposition}

\newtheorem{rem} [thm]{Remark}
\newtheorem{definition}[thm]{Definition}

\newtheorem{step}{Step}

\def\ph{{\varphi}}

\def\eps{\varepsilon}

\renewcommand\phi{\varphi}

\def\d{\textrm{d}}

\def\im{\textrm{im}\,}
\def\G{\mathrm{G}}

\def\SO{\mathbf{SO}}
\def\crit{\operatorname{Crit}}

\newcommand{\N}{\mathbb{N}}
\newcommand{\Z}{\mathbb{Z}}

\newcommand{\R}{\mathbb{R}}

\def\C{\mathbb C}

\def\d{{\operatorname{d}}}
\def\reg{{\operatorname{reg}}}

\def\D{\mathbb{D}}

\newcommand\hybr{\operatorname{hybr}}

\def\ind{\operatorname{ind}}

\newcommand\supp{\operatorname{supp}}

\def\ad{\operatorname{ad}}
\def\Crit{\operatorname{Crit}}
\def\d{\operatorname{d}}

\def\dom{\operatorname{dom}}
\def\im{\operatorname{im}}

\def\supp{\operatorname{supp}}

\def\A{\mathcal A}
\def\AA{\mathbbm A}

\def\D{\mathcal D}
\def\F{\mathcal F}
\def\G{\mathcal G}

\def\J{\mathcal J}

\def\YM{\mathcal{YM}}

\def\D{\mathbb{D}}

\def\coker{\operatorname{coker}}

\def\sign{\operatorname{sign}}
\def\dvol{\operatorname{dvol}}
\def\modulo{\operatorname{mod}}

\title{Elliptic Yang--Mills Flow Theory}
\author[R.~Janner]{R\'{e}mi Janner}
\address[R.~Janner]{Alpiq Management AG\\Bahnhofquai 12\\CH-4601 Olten\\Switzerland}
\email{remi.janner@gmail.com}
\urladdr{ www.remijanner.ch}

\author[J.~Swoboda]{Jan Swoboda}
\address[J.~Swoboda]{Max-Planck-Institut f\"ur Mathematik\\Vivatsgasse 7\\D-53111 Bonn\\Germany}
\email{swoboda@mpim-bonn.mpg.de}
\urladdr{http://www.mpim-bonn.mpg.de/de/node/94}

\date{\today}

\dedicatory{Dedicated to Professor Dietmar A.~Salamon on the occasion of his 60th birthday}
\subjclass[2010]{58E15, 53C07, 35J60, 53D20}

\begin{document}
\maketitle

\begin{abstract}  
We lay the foundations of a Morse homology on the space of connections on a principal $G$-bundle over a compact manifold $Y$, based on a newly defined gauge-invariant functional $\J$. While the critical points of $\J$ correspond to Yang--Mills connections on $P$, its $L^2$-gradient gives rise to a novel system of elliptic equations. This contrasts previous approaches to a study of the Yang--Mills functional via a parabolic gradient flow. We carry out the complete analytical details of our program in the case of a compact two-dimensional base manifold $Y$. We furthermore discuss its relation to the well-developed parabolic Morse homology of Riemannian surfaces. Finally, an application of our elliptic theory is given to three-dimensional product manifolds $Y=\Sigma\times S^1$.
\end{abstract}

\tableofcontents

\section{Introduction}\label{sec:main}
During the last decades many authors, in order to understand Morse theoretical properties of the Yang--Mills functional
\begin{eqnarray*}
\YM\colon\A(P)\to\R,\qquad \YM(A)=\frac{1}{2}\int_Y\langle F_A\wedge\ast F_A\rangle
\end{eqnarray*}
on a principal $G$-bundle $P$ over a compact manifold $Y$, studied its $L^2$-gradient flow  
\begin{equation}\label{yml2f}
\partial_sA+d_A^*F_A=0.
\end{equation}
This approach was introduced by Atiyah and Bott (cf.~\cite{AB}) and used for example by Donaldson (cf.~\cite{Donaldson}) to prove a generalized version of the Narasimhan--Seshadri theorem. Analytical properties of the solutions of \eqref{yml2f} for principal bundles over $2$- and $3$-dimensional base manifolds, or over base manifolds with a symmetry of codimension $3$, were proven by R\r{a}de (cf.~\cite{Rade,Rade2}) and Davis (cf.~\cite{Davies}). A consideration common to all of these works is the following one. Since Eq.~\eqref{yml2f} is invariant under an infinite-dimensional group of symmetries, the group $\G(P)$ of gauge transformations, it is only degenerate parabolic, in contrary to e.g.~the heat flow on manifolds. This problem can be overcome by imposing a suitable gauge-fixing condition. The linearized flow equation then splits into an operator of mixed parabolic and elliptic type (cf.~\cite{Davies,Janner,Swoboda1}). Short-time existence of the thus augmented equation is shown in \cite{Rade}. Long-time existence holds below the critical dimension $n=4$, and is unknown to be case for $n=4$ (cf.~Struwe's article \cite{Struwe} for a blow-up analysis in this case). Moreover, it is not known whether the flow satisfies the Morse--Smale transversality property and it is therefore in general an open question whether a Morse homology, based on the Yang--Mills functional and its $L^2$-flow, can be defined in dimensions $n\geq3$.\\
\noindent\\
In this context we present a novel approach to Yang--Mills theory and define a new Morse homology as follows. Let $Y$ be a closed manifold of dimension $n\geq2$. Fix a Riemannian metric $g$ on $Y$ and let $\dvol_Y$ denote its volume form. We consider a principal $G$-bundle $P\to Y$, $G$ a compact Lie group with Lie algebra $\mathfrak g$, and an $\ad$-invariant inner product on $\mathfrak g$. Let $X:=\mathcal A(P)\times \Omega^{n-2}(Y,\ad(P))$ be  the product of the space of $\mathfrak g$-valued connections on $P$ and $\ad(P)$-valued $(n-2)$-forms. Then we consider the energy functional
\begin{equation*}
\mathcal J\colon X\to\R,\qquad\mathcal J(A,\omega):=\int_Y \left(\langle F_A, \ast \omega \rangle-\frac 12|\omega|^2\right)\dvol_Y.
\end{equation*}
Critical points $(A,\omega)$ of this functional satisfy the system of equations
\begin{equation*}
\ast F_A-\omega=0,\qquad d_A\omega=0
\end{equation*}
in which case $A$ is a Yang--Mills connection (i.e., satisfies $d_A^{\ast}F_A=0$) and $\ast\omega$ is its curvature. The most interesting aspect are the $L^2$-gradient flow equations
\begin{equation}\label{newflow}
\partial_sA+(-1)^{n+1}\ast d_A\omega=0,\qquad\partial_s\omega+\ast F_A-\omega=0,
\end{equation}
a system of first order nonlinear PDEs, which for $n=2$ define a nice elliptic problem. Based on the gradient flow equations \eqref{newflow}, we define a new elliptic Yang--Mills Morse homology. To make this idea precise, we need to work with a perturbed version $\J+h_f$ of the functional $\J$. The reason for adding a perturbation $h_f$ is twofold. First, it is in general not the case that critical points of the unperturbed functional $\J$ are nondegenerate in the Morse--Bott sense. By this we mean that the nullspace of the Hessian $H_{(A,\omega)}\J$ of $\J$ at each critical point $(A,\omega)$ consists entirely of infinitesimal gauge transformations. The second reason is that a priori it is not clear whether gradient flow lines of $\J$, connecting two critical points, satisfy Morse--Smale transversality. This condition, which is equivalent to the linearized gradient flow equation at such a flow line being surjective, is required to obtain smooth moduli spaces of solutions. We here work with holonomy perturbations (cf.~\textsection \ref{subsec:holonomypert} for details) the construction of whose is based on work by Salamon and Wehrheim in \cite{SalWeh}. The precise form of the $L^2$-gradient flow equations resulting from the perturbed functional $\J+h_f$ will be stated in \eqref{pertEYM1} below.\\
\noindent\\
The plan of the article is as follows. In \textsection \ref{sec:equations} we introduce the elliptic Yang--Mills equations. We then move on to establish their main analytical properties as far as they are needed to define elliptic Yang--Mills homology in the case $n=2$ (which we introduce in Section \ref{sec:homology}). We in particular show (cf.~Section \ref{sect:modulispFredholm}) that the linearization of the gradient flow equations give rise to a Fredholm operator and determine its index. The main difficulty one encounters here is that the functional $\J$ is neither bounded from below nor from above so that the number of eigenvalue crossings of the resulting spectral flow cannot be read off from the index of the Hessian at limiting critical points of $\J$. To overcome this problem we relate this spectral flow to that of a further family of elliptic operators, the numbers of negative eigenvalues of their limits as $s\to\pm\infty$ being finite and equal to the index of the limiting Yang--Mills connections. A further result concerns compactness up to gauge transformations and convergence to broken trajectories of the moduli spaces of solutions to \eqref{pertEYM1} which have uniformly bounded energy, cf.~\textsection \ref{sect:compactness}. Exponential decay of finite energy solutions of \eqref{pertEYM1} towards critical points is shown in \textsection \ref{sect:exponentialdecay}. Transversality is being dealt with in \textsection \ref{sect:transversality}. A Morse boundary operator defined by counting the numbers of elements in suitable moduli spaces of connecting flow lines is introduced in \textsection \ref{sec:homology}. Together with the critical points of $\J+h_f$ in a fixed sublevel set it gives rise to a chain complex, the homology of which we name {\bf{elliptic Yang--Mills Morse homology}}.\\
\noindent\\
In \textsection \ref{sect:RelshipwithparabolicYM} we initiate a comparison of elliptic Yang--Mills Morse homology with its classical parabolic counterpart (cf.~\cite{AB,Swoboda1}). Although these theories are based on very different types of equations, they share a number of common properties. Namely, the set of generators of the relevant chain complexes is the same in both cases, and also the dimensions of moduli spaces of connecting flow lines agree in each case. Therefore it seems natural to conjecture that there exists an isomorphism between elliptic and parabolic Morse homology. We give some evidence to this conjecture by constructing a so-called hybrid moduli space in which we combine elliptic and parabolic flow lines. Leaving the analytical details of that construction to a future publication we anticipate, based on a crucial energy inequality, that arguments similar to those in \cite{AbSchwarz,AbSchwarz1,Swoboda2} apply and give rise to an invertible chain homomorphism between the parabolic and the elliptic Morse complex. \\
\noindent\\
As an application of the elliptic Yang--Mills homology presented here we consider in \textsection \ref{sec:special case} three-dimensional product manifolds $Y=\Sigma\times S^1$. We relate the new invariants obtained in this case to various other homology groups, amongst them Floer homology of the cotangent bundle of the space of gauge equivalence classes of flat $\SO(3)$ connections over $\Sigma$.\\
\noindent\\
One difficulty in extending elliptic Morse homology to manifolds $Y$ of dimension $n\geq3$ consists in the fact that the linearization of equation \eqref{newflow} ceases to be elliptic, even if appropriate gauge-fixing conditions are imposed. To overcome this problem we introduce in \textsection \ref{sec:restrflow} a modification of the above setup. The main idea is to restrict the configuration space $X=\mathcal A(P)\times \Omega^{n-2}(Y,\ad(P))$ to the Banach submanifold 
\begin{eqnarray*}
X_1:=\{(A,\omega)\in X\mid d_A^{\ast}\omega=0\}
\end{eqnarray*}
of $X$ and consider the flow \eqref{newflow} on $X_1$ instead. This modification is natural because the critical points of $J$ are automatically contained in $X_1$. As it turns out, the linearization obtained through this modification are in fact elliptic equations, however of nonlocal type. A discussion of their compactness and transversality properties is left to future work.

\subsection*{Acknowledgements} The second named author gratefully acknowledges DFG for its financial support through grant SW 161/1-1. He also thanks the Department of Mathematics of Stanford University for its hospitality where part of this work has been carried out.

\section{Elliptic Yang--Mills flow equations}\label{sec:equations}

Let $Y$ be a closed oriented smooth manifold of dimension $n\geq2$, endowed with a Riemannian metric $g$. We keep the notation as introduced before. For pairs $(A,\omega)\in X=\A(P)\times\Omega^{n-2}(Y,\ad(P))$ we consider the functional
\begin{eqnarray*}
\J\colon X\to\R,\qquad\J(A,\omega)=\int_Y\left(\langle F_A,\ast\omega\rangle-\frac{1}{2}|\omega|^2\right)\,\dvol_Y.
\end{eqnarray*}
For transversality reasons apparent later we add a so-called holonomy perturbation $h_f$ to $\J$ which we shall define in \textsection \ref{subsec:holonomypert}. The $L^2$-gradient of $\J+h_f$ is given by
\begin{eqnarray*}
\nabla\J(A,\omega)=\big((-1)^{n+1}\ast d_A\omega-X_f(A,\omega),\ast F_A-\omega+Y_f(A,\omega)\big).
\end{eqnarray*}
For a definition of the term $(X_f(A,\omega),Y_f(A,\omega))$ which results as the $L^2$-gradient of $h_f$, cf.~\eqref{eq:gradienthf} below. We define the {\bf{perturbed elliptic Yang--Mills flow}} to be the system of equations
\begin{eqnarray}\label{pertEYM1}
\begin{cases}
0=\partial_sA-d_A\Psi+(-1)^{n+1}\ast d_A\omega-X_f(A,\omega),\\
0=\partial_s\omega+[\Psi,\omega]-\omega+\ast  F_A+Y_f(A,\omega).
\end{cases} 
\end{eqnarray}
for a connection $A\in\A(P)$ and $\ad(P)$-valued forms $\Psi\in\Omega^0(Y,\ad(P))$ and $\omega\in\Omega^{n-2}(Y,\ad(P))$. The term $\Psi$ has been introduced in \eqref{pertEYM1} in order to make this system of equations invariant under the action of time-dependent gauge transformations. Stationary points of the unperturbed flow equation (referring to the unperturbed functional $\J$) such that $\Psi=0$ satisfy 
\begin{eqnarray}\label{statEYM}
d_A\omega=0\qquad\textrm{and}\qquad\omega=\ast F_A.
\end{eqnarray}
Then $d_A\omega=d_A\ast F_A=0$ and hence the set of critical points of $\J$ is in bijection with the set of Yang--Mills connections on the bundle $P$.

\begin{rem}\upshape
The factor $(-1)^{n+1}$ appearing in the first equation in \eqref{pertEYM1} results from the different signs the formal adjoint $d_A^{\ast}=(-1)^{n(\deg\omega+1)+1}\ast d_A\ast\omega$ of $d_A$ has in different degrees and dimensions.
\end{rem}

\begin{rem}\upshape
Assume $(A,\omega)$ is a solution of the unperturbed equation \eqref{pertEYM1} on $I\times Y$, $I$ a bounded or unbounded interval. Assume in addition that $\Psi=0$. Then it follows that $\omega$ satisfies the linear PDE
\begin{eqnarray}\label{eq:omegasecorder}
0=\ddot\omega-\dot\omega-d_A^{\ast}d_A\omega.
\end{eqnarray} 
In the case  $n=2$ this is an elliptic equation with the Hodge Laplacian $-\frac{d^2}{ds^2}+\Delta_A$ acting on $\Omega^0(I\times Y,\ad(P))$ as leading order term. 
\end{rem}

\section{Moduli spaces and Fredholm theory}\label{sect:modulispFredholm}

From now on we specialize to the case $n=2$ and let $Y=\Sigma$ denote a closed oriented surface, being endowed with a Riemannian metric $g$.

\subsection{Gradient flow lines and moduli spaces}

Throughout we call a smooth solution $(A,\omega,\Psi)$ of \eqref{pertEYM1} on $\R\times\Sigma$ a {\bf{negative gradient flow line}} of $\mathcal J$. We say that $(A,\omega,\Psi)$ is in {\bf{temporal gauge}} if $\Psi=0$. The {\bf{energy}} of a solution $(A,\omega,\Psi)$ of \eqref{pertEYM1} is  
\begin{multline*}
E_f(A,\omega,\Psi):=\\
\frac{1}{2}\int_{\R}\|\ast d_A\omega+d_A\Psi+X_f(A,\omega)\|_{L^2(\Sigma)}^2+\|\ast F_A-\omega+[\Psi\wedge\omega]+Y_f(A,\omega)\|_{L^2(\Sigma)}^2\,ds.
\end{multline*}
It is by definition invariant under time-dependent gauge transformations. In \textsection \ref{sect:exponentialdecay} it will be shown that a temporally gauged solution $(A,\omega,0)$ of \eqref{pertEYM1}  on $\R\times\Sigma$ has finite energy if and only if there exist critical points $(A^{\pm},\omega^{\pm})$ of $\mathcal J+h_f$ such that $(A(s),\omega(s))$ converges exponentially to $(A^{\pm},\omega^{\pm})$ as $s\to\pm\infty$. We let $\mathcal M_f(A^-,\omega^-,A^+,\omega^+)$ denote the moduli space of gauge equivalence classes of negative gradient flow lines from $(A^-,\omega^-)$ to $(A^+,\omega^+)$, i.e.~
\begin{eqnarray*}
\mathcal M_f(A^-,\omega^-,A^+,\omega^+):=\frac{\widehat{\mathcal M}_f(A^-,\omega^-,A^+,\omega^+)}{\G(P)},
\end{eqnarray*}
where
\begin{eqnarray*}
\widehat{\mathcal M}_f(A^-,\omega^-,A^+,\omega^+):=\left\{(A,\omega)\left|\begin{array}{c}(A,\omega,0)\;\textrm{satisfies}\;\eqref{pertEYM1} ,E_f(A,\omega,0)<\infty,\\\lim_{s\to\pm\infty}(A(s),\omega(s))\in[(A^{\pm},\omega^{\pm})]\end{array}\right.\right\}.
\end{eqnarray*}
The moduli space $\mathcal M_f(A^-,\omega^-,A^+,\omega^+)$ can be described in a slightly different way as the quotient of all finite energy gradient flow lines $(A,\omega,\Psi)$ from $(A^-,\omega^-)$ to $(A^+,\omega^+)$ such that $\Psi(s)\to0$ as $s\to\pm\infty$ modulo the action of the group $\mathcal G(\R\times\Sigma)$ of smooth time-dependent gauge transformations, which converge exponentially to the identity as $s\to\pm\infty$. One of our goals in the subsequent sections is to show that the moduli space $\mathcal M_f(A^-,\omega^-,A^+,\omega^+)$ is a compact manifold with boundary and to determine its dimension. We start by considering first the linearized operator for Eq.~\eqref{pertEYM1}.

\subsection{Linearized operator}

Linearizing the map
\begin{eqnarray*}
\F\colon(A,\omega,\Psi)\mapsto\left(\begin{array}{c}\partial_sA-d_A\Psi-\ast d_A\omega-X_f(A,\omega)\\\partial_s\omega+[\Psi,\omega]-\omega+\ast  F_A+Y_f(A,\omega)\end{array}\right)
\end{eqnarray*}
at $(A,\omega,\Psi)$ yields the linear operator 
\begin{multline}\label{linoperator}
\mathcal D_{(A,\omega,\Psi)}\colon(\alpha,v,\psi)\mapsto\\
\nabla_s\left(\begin{array}{c}\alpha\\v\\\psi\end{array}\right)+\underbrace{\left(\begin{array}{ccc}\ast[\omega\wedge\,\cdot\,]&-\ast d_A&-d_A\\\ast d_A&-\mathbbm 1&-[\omega\wedge\,\cdot\,]\\-d_A^{\ast}&[\omega\wedge\,\cdot\,]&0\end{array}\right)}_{=:B_{(A,\omega,\Psi)}}\left(\begin{array}{c}\alpha\\v\\\psi\end{array}\right)+\begin{pmatrix}-\d\!X_f(A,\omega)(\alpha,v)\\\d\!Y_f(A,\omega)(\alpha,v)\\0\end{pmatrix},
\end{multline}
where $\nabla_s:= \frac\partial{\partial_s}+[\Psi \wedge \,\cdot\,]$. To specify its domain and target we define for $p>1$
\begin{eqnarray*}
\mathcal L^p:=L^p(\R,L^p(\Sigma))\qquad\textrm{and}\qquad\mathcal W^p:=W^{1,p}(\R,L^p(\Sigma))\cap L^p(\R,W^{1,p}(\Sigma)),
\end{eqnarray*}
where for abbreviation we set
\begin{multline*}
L^p(\R,L^p(\Sigma)):=L^p(\R,L^p(\Sigma,T^{\ast}\Sigma\otimes\ad(P)))\\
\oplus L^p(\R,L^p(\Sigma,\ad(P)))\oplus L^p(\R,L^p(\Sigma,\ad(P))),
\end{multline*}
and similarly for $L^p(\R,W^{1,p}(\Sigma))$ and $W^{1,p}(\R,L^p(\Sigma))$. Throughout we shall consider $\mathcal D_{(A,\omega,\Psi)}$ as an operator $\mathcal D_{(A,\omega,\Psi)}\colon\mathcal W^p\to\mathcal L^p$.

\begin{rem}[Gauge-fixing condition]\label{rem:gaugefix}\upshape
The last component in the definition \eqref{linoperator} of the operator $\mathcal D_{(A,\omega,\Psi)}$ is a gauge-fixing condition. It can be understood as follows. We identify the pair $(A,\Psi)$ with the connection $A+\Psi\,ds$ on the principal $G$-bundle $\R\times P$. The group $\G(\R\times P)$ of time-dependent gauge transformations acts on pairs $(A+\Psi\,ds,\omega)$ as
\begin{eqnarray*}
g^{\ast}(A+\Psi\,ds,\omega)=(g^{\ast}A+(g^{-1}\Psi g+g^{-1}\dot g)\wedge ds,g^{-1}\omega g).
\end{eqnarray*}
The infinitesimal action at $(A+\Psi\,ds,\omega)$ is the map
\begin{eqnarray*}
\ph\mapsto(d_A\ph+(\dot\ph+[\Psi\wedge\ph])\wedge ds,[\omega\wedge\ph]).
\end{eqnarray*}
A short calculation now shows that $(\alpha+\psi\wedge ds,v)$ is orthogonal to the image of the infinitesimal action precisely if the last component of $\mathcal D_{(A,\omega,\Psi)}(\alpha,v,\psi)$ vanishes.
\end{rem}

\subsection{Fredholm theorem} 
Throughout this section we fix $a>0$ and an $a$-regular perturbation $h_f$. The aim of this section is to prove the following theorem.

\begin{thm}\label{thm:Fredholmthm}
Let $(A,\omega,\Psi)$ be a solution of $\eqref{pertEYM1}$ and assume that there exist solutions $(A^{\pm},\omega^{\pm})$ of the perturbed critical point equation \eqref{eq:pertcritpointeq} such that
\begin{eqnarray*}
\lim_{s\to\pm\infty}(A(s),\omega(s),\Psi(s))=(A^{\pm},\omega^{\pm},0)
\end{eqnarray*}
in $C^k(\Sigma)$ for every $k\in\N_0$. Then for every $1<p<\infty$ the linear operator $\mathcal D_{(A,\omega,\Psi)}\colon\mathcal W^p\to\mathcal L^p$ is a Fredholm operator of index
\begin{eqnarray}\label{eq:Fredholmindex}
\ind\mathcal D_{(A,\omega,\Psi)}=\ind H_{A^-,f}-\ind H_{A^+,f}.
\end{eqnarray}
Here we denote by $\ind H_{A,f}$ the Morse index (i.e.~the number of negative eigenvalues) of the perturbed Yang--Mills Hessian 
\begin{eqnarray*}
H_{A,f}\coloneqq d_A^{\ast}d_A+\ast[\ast F_A\wedge\,\cdot\,]+\d\!X_f(A).
\end{eqnarray*}
\end{thm}

Because the statement of this theorem is invariant under gauge transformations we may from now on assume that $\Psi=0$. For ease of notation we also define $B_f\coloneqq B_{(A,\omega,0)}+(-\d\!X_f(A,\omega),\d\!Y_f(A,\omega))$.

\subsubsection*{Fredholm property and index in the case $p=2$}

We let $p=2$. Then $B_f(s)$ is a self-adjoint operator on the Hilbert space $H:=L^2(\Sigma)$ with domain $W\coloneqq\dom B_f(s)=W^{1,2}(\Sigma)$, for every $s\in\R$. This follows by the same arguments as given in the proof of Proposition \ref{prop:ellipticsystem}.

\begin{proof}{\bf{[Theorem \ref{thm:Fredholmthm} in the case $p=2$].}}
To show the Fredholm property we employ the result \cite[Theorem A]{RobSal}. To apply this theorem we need to check that the following conditions (i-iv) are satisfied. (i) The inclusion $W\hookrightarrow H$ of Hilbert spaces is compact with dense range. This holds true by the Rellich--Kontrachov compactness theorem. (ii) The norm of $W$ is equivalent to the graph norm of $B_f(s)\colon W\to H$ for every $s\in\R$. This follows from a standard elliptic estimate for the operator $B_f(s)\colon W\to H$. (iii) The map $\R\to\mathcal L(W,H)\colon s\mapsto B_f(s)$ is continuously differentiable with respect to the weak operator topology. For this we need to verify that for every $\xi\in W$ and $\eta\in H$ the map $s\mapsto\langle B_f(s)\xi,\eta\rangle$ is of class $C^1(\R)$. Because $s\mapsto(A(s),\omega(s))$ is a smooth path in $X$ and $B_f$ depends smoothly on $(A,\omega)$, this property is clearly satisfied. (iv) The operators $B_{(A^{\pm},\omega^{\pm},0)}-(\d\!X_f(A^{\pm}),0)$ are invertible and are the limits of $B_f(s)$ in the norm topology as $s\to\pm\infty$. Invertibility follows because the perturbation $h_f$ was assumed to be $a$-regular. Theorem \ref{thm:mainexpconvergence} (exponential decay) gives uniform convergence with all derivatives of $(A(s),\omega(s))$ to $(A^{\pm},\omega^{\pm})$ as $s\to\pm\infty$, hence in particular norm convergence $B_f(s)\to B_{(A^{\pm},\omega^{\pm},0)}-(\d\!X_f(A^{\pm}),0)$ as $s\to\pm\infty$. We have thus verified that the operator family $s\mapsto B_f(s)$ satisfies all the assumptions of \cite[Theorem A]{RobSal}. It hence follows that the operator $\mathcal D_{(A,\omega,0)}=\frac{d}{ds}+B_f(s)$ is Fredholm with index given by the spectral flow of the family $s\mapsto B_f(s)$, cf.~below. It is shown in Proposition \ref{prop:interpolatingfamily} and Lemma \ref{lem:eqcrosssign} below that this spectral flow equals that of a further operator family $s\mapsto C_f(s)$. Its spectral flow in turn is equal to the right-hand side of \eqref{eq:Fredholmindex}, which is the content of the subsequent Lemma \ref{lem:crossingformula}. The index formula \eqref{eq:Fredholmindex} now follows.
\end{proof}

It remains to determine the index of the Fredholm operator $\mathcal D_{(A,\omega,0)}$. This index is related to the spectral flow of the operator family $s\mapsto B_f(s)$ as we explain next. Here we follow the discussion in \cite[\textsection 4]{RobSal}. Recall that a {\bf{crossing}} of $B_f$ is a number $s\in\R$ for which $B_f(s)$ is not injective. The {\bf{crossing operator}} at $s\in\R$ is the map
\begin{eqnarray*}
\Gamma(B_f,s)\colon\ker B_f(s)\to\ker B_f(s),\qquad\Gamma(B_f,s)=P(s)\dot B_f(s),
\end{eqnarray*}
where $P(s)\colon H\to H$ denotes the orthogonal projection onto $\ker B_f(s)$. A crossing $s\in\R$ is called {\bf{regular}} if $\Gamma(B_f(s))$ is nonsingular. It can be shown (by adapting the arguments in \cite[Theorem 4.2]{RobSal}) that, after replacing $f$ by a suitable, $a$-regular perturbation $\hat f$ (which may be chosen arbitrarily close to $f$), the operator family $s\mapsto B_{\hat f}(s)$ only has regular crossings. We assume that $f$ has been chosen accordingly. Then the number of crossings of $B_f$ is in particular finite. The {\bf{signature}} of the crossing $s\in\R$ is the signature (the number of positive minus the number of negative eigenvalues) of the endomorphism $\Gamma(B_f,s)$. The Fredholm index of $\mathcal D_{(A,\omega,0)}$ is determined by the crossing signatures via the relation
\begin{eqnarray}\label{eq:formspectralflow}
\ind\mathcal D_{(A,\omega,0)}=-\sum_s\sign\Gamma(B_f,s)
\end{eqnarray}
where the sum is over all crossings, cf.~\cite[Theorem 4.1]{RobSal}.\\
\noindent\\
We now aim to determine the Fredholm index of $\mathcal D_{(A,\omega,0)}$ from \eqref{eq:formspectralflow}. The following proposition simplifies the subsequent computations.

\begin{prop}\label{prop:interpolatingfamily}
Let $B_{f,0}$ denote the family of operators which we obtain from $B_f$ by setting $Y=0$. Then the operator $\frac{d}{ds}+B_{f,0}\colon\mathcal W^2\to\mathcal L^2$ is a Fredholm operator with index equal to that of $\mathcal D_{(A,\omega,0)}$.
\end{prop}

\begin{proof}
For $0\leq\tau\leq1$ consider the family of operators $\frac{d}{ds}+B_{(A,\omega,0)}+(-\d\!X_f(A,\omega),\tau\d\!Y_f(A,\omega))$. Since $Y_f(A(s),\omega(s))=0$ for sufficiently large $|s|\geq s_0$ (which holds by our choice of the $a$-regular perturbation $f$) it follows that for each $\tau$ the limiting operator as $s\to\pm\infty$ of $B_{(A,\omega,0)}+(-\d\!X_f(A,\omega),\tau\d\!Y_f(A,\omega))$ is equal to the invertible operator $B_{(A^{\pm},\omega^{\pm},0)}-(\d\!X_f(A^{\pm},0)$. By the same proof as of Theorem \ref{thm:Fredholmthm} in the case $p=2$ this implies that each of the operators in the  above family is Fredholm. Because this family is an interpolating family of Fredholm operators between $\frac{d}{ds}+B_{f,0}$ and $\mathcal D_{(A,\omega,0)}$ equality of the indices follows.
\end{proof}

To determine the Fredholm index of $\mathcal D_{(A,\omega,0)}$ we use Proposition \ref{prop:interpolatingfamily} together with relation \eqref{eq:formspectralflow} (with the operator $B_f$ replaced by $B_{f,0}$). The crossing indices of the operator family $B_{f,0}$ are found by relating them to those of a further path of operators $C_f(s)$ which we introduce next. For $\lambda\in\R\setminus\{-1\}$ and $(A,\omega)=(A(s),\omega(s))$ we define the symmetric operator $C_{f,\lambda}(s)$ by
\begin{eqnarray*}
C_{f,\lambda}(s):=\left(\begin{array}{cc}\frac{1}{\lambda+1}d_A^{\ast}d_A+\ast[\omega\wedge\,\cdot\,]+\d\!X_f(A,\omega)&-d_A+\frac{1}{\lambda+1}\ast d_A[\omega\wedge\,\cdot\,]\\
-d_A^{\ast}+\frac{1}{\lambda+1}\ast[\omega\wedge d_A\,\cdot\,]&-\frac{1}{\lambda+1}\ast[\omega\wedge\ast[\omega\wedge\,\cdot\,]]\end{array}\right).
\end{eqnarray*}
We furthermore set $C_f(s):=C_{f,0}(s)$.

\begin{lem}\label{lem:eqcrosssign}
For every $s\in\R$ the crossing indices $\Gamma(B_{f,0}(s))$ and $\Gamma(C_f(s))$ coincide.
\end{lem}

\begin{proof}
We first note that $\lambda=0$ is an eigenvalue of $B_{f,0}(s)$ if and only if it is an eigenvalue of $C_f(s)$. In this case, the corresponding eigenspaces are of the same dimension. These facts follow from Proposition \ref{prop:redeigenvalueeq} below. It remains to show that for every crossing $s_0\in\R$ the signatures $\sign\Gamma(B_{f,0},s_0)$ and $\sign\Gamma(C_f,s_0)$ coincide. We prove this equality for a simple crossing $s_0$, i.e.~in the case where the crossing is regular and in addition the kernels of $B_{f,0}(s_0)$ and $C_f(s_0)$ are one-dimensional. In this situation let $s\mapsto\lambda(s)$, $s\in(s_0-\eps,s_0+\eps)$, be a $C^1$-path of eigenvalues of $B_{f,0}(s)$ with corresponding path of eigenvectors $\xi(s)=(\alpha(s),v(s),\psi(s))$ such that $\lambda(s_0)=0$. Similarly, let $s\mapsto\mu(s)$, $s\in(s_0-\eps,s_0+\eps)$, be a $C^1$-path of eigenvalues of $C_f(s)$ with normalized eigenvectors $\tilde\xi(s)=(\tilde\alpha(s),\tilde\psi(s))$ such that $\mu(s_0)=0$ and $\tilde\xi(s_0)=(\alpha(s_0),\psi(s_0))$. It suffices to prove that
\begin{eqnarray}\label{eq:claimsign}
\sign\dot\lambda(s_0)=\sign\dot\mu(s_0).
\end{eqnarray}
To prove this claim we first apply Proposition \ref{prop:redeigenvalueeq} below which gives the identity
\begin{eqnarray*}
C_{f,\lambda(s)}(s)(\alpha(s),\psi(s))=\lambda(s)(\alpha(s),\psi(s))
\end{eqnarray*}
for all $s\in(s_0-\eps,s_0+\eps)$. Next, by definition of the operator families $C_{f,\lambda(s)}(s)$ and $C_f(s)$ it follows that
\begin{multline}\label{eq:diffeigenCC} 
\left.\frac{d}{ds}\right|_{s=s_0}C_{f,\lambda(s)}(s)=\dot C_f(s_0)+\dot\lambda(s_0)\left(\begin{array}{cc}-d_A^{\ast}d_A&-\ast d_A[\omega\wedge\,\cdot\,]\\-\ast[\omega\wedge d_A\,\cdot\,]&\ast[\omega\wedge\ast[\omega\wedge\,\cdot\,]]\end{array}\right).
\end{multline}
We apply Proposition \ref{prop:vareigenvalue} below to the operator family $s\mapsto C_{f,\lambda(s)}(s)$.  Together with \eqref{eq:diffeigenCC} it then follows that (we abbreviate $A:=A(s_0)$, $\omega:=\omega(s_0)$ and recall that $\tilde\xi(s_0)=(\alpha(s_0),\psi(s_0))$)
\begin{align*}
\dot\lambda(s_0)=&\Big\langle\left.\frac{d}{ds}\right|_{s=s_0}C_{f,\lambda(s)}(s)\tilde\xi(s_0),\tilde\xi(s_0)\Big\rangle\\
=&\langle\dot C_f(s_0)\tilde\xi(s_0),\tilde\xi(s_0)\rangle\\
&+\dot\lambda(s_0)\left\langle\left(\begin{array}{cc}-d_A^{\ast}d_A&-\ast d_A[\omega\wedge\,\cdot\,]\\-\ast[\omega\wedge d_A\,\cdot\,]&\ast[\omega\wedge\ast[\omega\wedge\,\cdot\,]]\end{array}\right)\tilde\xi(s_0),\tilde\xi(s_0)\right\rangle\\
=&\dot\mu(s_0)-\dot\lambda(s_0)\langle d_A^{\ast}d_A\alpha(s_0),\alpha(s_0)\rangle+2\dot\lambda(s_0)\langle\ast d_A\alpha(s_0),[\omega\wedge\psi(s_0)]\rangle\\
&+\dot\lambda(s_0)\langle\ast[\omega\wedge\ast[\omega\wedge\psi(s_0)],\psi(s_0)\rangle\\
=&\dot\mu(s_0)-\dot\lambda(s_0)\|d_A\alpha(s_0)\|^2-\dot\lambda(s_0)\|[\omega\wedge\psi(s_0)]\|^2\\
&+2\dot\lambda(s_0)\langle\ast d_A\alpha(s_0),[\omega\wedge\psi(s_0)]\rangle.
\end{align*}
Here we used again Proposition \ref{prop:vareigenvalue} in order to replace $\langle\dot C_f(s_0)\tilde\xi(s_0),\tilde\xi(s_0)\rangle=\dot\mu(s_0)$. Concerning the remaining terms in the same line we compute 
\begin{align*}
\langle\alpha(s_0),-\ast d_A[\omega\wedge\psi(s_0)]\rangle=&\langle\ast\alpha(s_0),d_A[\omega\wedge\psi(s_0)]\rangle\\
=&\langle d_A^{\ast}\ast\alpha(s_0),[\omega\wedge\psi(s_0)]\rangle\\
=&\langle\ast d_A\alpha(s_0),[\omega\wedge\psi(s_0)]\rangle,
\end{align*}
and likewise for $\langle-\ast[\omega\wedge d_A\alpha(s_0)],\psi(s_0)\rangle$. It finally follows that
\begin{multline*}
\dot\mu(s_0)=\big(1+\|d_A\alpha(s_0)\|^2+\|[\omega\wedge\psi(s_0)]\|^2-2\langle\ast d_A\alpha(s_0),[\omega\wedge\psi(s_0)]\rangle\big)\dot\lambda(s_0).
\end{multline*}
Using the Cauchy--Schwartz inequality we can estimate the factor in front of $\lambda(s_0)$ in the previous line as
\begin{align*}
&1+\|d_A\alpha(s_0)\|^2+\|[\omega\wedge\psi(s_0)]\|^2-2\langle\ast d_A\alpha(s_0),[\omega\wedge\psi(s_0)]\rangle\\
\geq&1+\|d_A\alpha(s_0)\|^2+\|[\omega\wedge\psi(s_0)]\|^2-2\|d_A\alpha(s_0)\|\|[\omega\wedge\psi(s_0)]\|\\
\geq&1.
\end{align*}
Therefore, $\dot\mu(s_0)$ and $\dot\lambda(s_0)$ have the same sign which shows \eqref{eq:claimsign} in the case of a simple crossing $s_0$. The case of a general regular crossing can be treated similarly. This completes the proof.
\end{proof}

\begin{prop}\label{prop:redeigenvalueeq}
Let the operators $B_{f,0}(s)$ and $C_{f,\lambda}(s)$, $s\in\R$, be as defined above. Let $\lambda\in\R\setminus\{-1\}$. Then $\xi=(\alpha,v,\psi)\in\dom B_{f,0}(s)$ satisfies the eigenvalue equation $B_{f,0}(s)\xi=\lambda\xi$ if and only if $\tilde\xi=(\alpha,\psi)$ satisfies the nonlinear eigenvalue equation
\begin{eqnarray}\label{eq:evC}
C_{f,\lambda}(s)\tilde\xi=\lambda\tilde\xi.
\end{eqnarray} 
\end{prop}

\begin{proof}
Assume $\xi=(\alpha,v,\psi)\in\dom B_{f,0}(s)$ satisfies
\begin{eqnarray}\label{eq:evB}
B_{f,0}(s)\xi=\lambda\xi
\end{eqnarray}
for some $\lambda\neq-1$. Then it follows from the definition of $B_{f,0}(s)$ that 
\begin{eqnarray}\label{eq:eigenvv}
v=\frac{1}{\lambda+1}(\ast d_A\alpha-[\omega\wedge\psi]).
\end{eqnarray}
Inserting this $v$ into the first and last of the three equations in \eqref{eq:evB} yields \eqref{eq:evC}. Conversely, assume that $\tilde\xi=(\alpha,\psi)$ satisfies \eqref{eq:evC} for some $\lambda\neq-1$. Defining $v$ by equation \eqref{eq:eigenvv} and setting $\xi:=(\alpha,v,\psi)$ we obtain a solution $\xi$ of \eqref{eq:evB} for this eigenvalue $\lambda$.
\end{proof}

\begin{prop}\label{prop:vareigenvalue}
Let $s\mapsto F(s)$, $s\in(s_0-\eps,s_0+\eps)$, be a $C^1$-path of densely defined symmetric operators on a Hilbert space $H$. Let $s\mapsto\lambda(s)$, $s\in(s_0-\eps,s_0+\eps)$, be a $C^1$-path of eigenvalues of the operator family $F$ with normalized eigenvectors $x(s)$. Then it follows that
\begin{eqnarray*}
\dot\lambda(s_0)=\langle\dot F(s_0)x(s_0),x(s_0)\rangle.
\end{eqnarray*}
\end{prop}

\begin{proof}
We differentiate the eigenvalue equation $F(s)x(s)=\lambda(s)x(s)$ at $s_0$ and obtain
\begin{eqnarray*}
\dot F(s_0)x(s_0)+F(s_0)\dot x(s_0)=\dot\lambda(s_0)x(s_0)+\lambda(s_0)\dot x(s_0).
\end{eqnarray*}
Take the inner product of both sides with $x(s_0)$ and use that by symmetry of $F(s_0)$
\begin{eqnarray*}
\langle F(s_0)\dot x(s_0),x(s_0)\rangle=\langle\dot x(s_0),F(s_0)x(s_0)\rangle=\langle\dot x(s_0),\lambda(s_0)x(s_0)\rangle
\end{eqnarray*}
to conclude the result.
\end{proof}

\begin{lem}\label{lem:crossingformula}
The total number $\sum_{s\in\R}\Gamma(C_f(s))$ of eigenvalue crossings of the operator family $C_f$ equals the right-hand side of \eqref{eq:Fredholmindex}. 
\end{lem}

\begin{proof}
We set $C_f^{\pm}:=\lim_{s\to\pm\infty}C_f(s)$. By assumption, each pair $(A^{\pm},\omega^{\pm})=\lim_{s\to\pm\infty}(A(s),\omega(s))$ satisfies the critical point equations $\omega^{\pm}=\ast F_{A^{\pm}}$ and $\ast d_{A^{\pm}}\omega^{\pm}+X_f(A^{\pm})=0$. Inserting the first one into the definition of $C_f$ it follows that  
\begin{eqnarray*}
C_f^{\pm}=\left(\begin{array}{cc}d_{A^{\pm}}^{\ast}d_{A^{\pm}}+\ast[\ast F_{A^{\pm}}\wedge\,\cdot\,]-\d\!X_f(A^{\pm})&-d_{A^{\pm}}+\ast d_{A^{\pm}}[\ast F_{A^{\pm}}\wedge\,\cdot\,]\\
-d_{A^{\pm}}^{\ast}+\ast[\ast F_{A^{\pm}}\wedge d_{A^{\pm}}\,\cdot\,]&-\ast[\ast F_{A^{\pm}}\wedge[F_{A^{\pm}}\wedge\,\cdot\,]]\end{array}\right).
\end{eqnarray*}
We point out that the upper left entry of $C_f^{\pm}$ equals the perturbed Yang--Mills Hessian $H_{A^{\pm},f}$. For $0\leq\tau\leq1$ we consider the smooth path 
\begin{multline*}
C_{\tau,f}^{\pm}\coloneqq\left(\begin{array}{cc}H_{A^{\pm},f}&-d_{A^{\pm}}+\tau\ast d_{A^{\pm}}[\ast F_{A^{\pm}}\wedge\,\cdot\,]\\
-d_{A^{\pm}}^{\ast}+\tau\ast[\ast F_{A^{\pm}}\wedge d_{A^{\pm}}\,\cdot\,]&-\tau\ast[\ast F_{A^{\pm}}\wedge[F_{A^{\pm}}\wedge\,\cdot\,]]\end{array}\right)\colon \\
W^{2,2}(\Sigma)\oplus W^{1,2}(\Sigma)\to L^2(\Sigma)\oplus L^2(\Sigma).
\end{multline*}
of bounded symmetric operators. Note that $C_{0,f}^{\pm}$ equals the perturbed Yang--Mills Hessian $H_{A^{\pm},f}$, augmented by a gauge-fixing condition, while $C_{1,f}^{\pm}=C_f^{\pm}$. We claim that $0$ is not an eigenvalue of $C_{\tau,f}^{\pm}$, for every $0\leq\tau\leq1$. To prove this claim, assume that $(\alpha,\psi)$ satisfies $C_{\tau,f}^{\pm}(\alpha,\psi)=0$. Then it follows that 
\begin{eqnarray*}
H_{A^{\pm},f}\alpha-d_{A^{\pm}}\psi+\tau\ast d_{A^{\pm}}[\ast F_{A^{\pm}}\wedge\psi]=0.
\end{eqnarray*}
Applying $d_{A^{\pm}}^{\ast}$ to both sides of the last equation yields
\begin{eqnarray*}
d_{A^{\pm}}^{\ast}d_{A^{\pm}}\psi-\tau d_{A^{\pm}}^{\ast}\ast d_{A^{\pm}}[\ast F_{A^{\pm}}\wedge\psi]=0.
\end{eqnarray*}
This follows because $\im H_{A^{\pm},f}\subseteq\ker d_{A^{\pm}}^{\ast}$. Take the $L^2$-inner product with $\psi$ to obtain
\begin{align*}
0&=\langle\psi,d_{A^{\pm}}^{\ast}d_{A^{\pm}}\psi-\tau d_{A^{\pm}}^{\ast}\ast d_{A^{\pm}}[\ast F_{A^{\pm}}\wedge\psi]\rangle\\
&=\langle d_{A^{\pm}}\psi,d_{A^{\pm}}\psi\rangle-\tau\langle\psi,\ast  d_{A^{\pm}}d_{A^{\pm}}[\ast F_{A^{\pm}}\wedge\psi]\rangle\\
&=\langle d_{A^{\pm}}\psi,d_{A^{\pm}}\psi\rangle-\tau\langle\psi,\ast[F_{A^{\pm}}\wedge[\ast F_{A^{\pm}}\wedge\psi]\rangle\\
&=\|d_{A^{\pm}}\psi\|^2+\tau\|[F_{A^{\pm}}\wedge\psi]\|^2.
\end{align*}
It follows that $\psi=0$ because the connections $A^{\pm}$ are irreducible. We hence conclude that $H_{A^{\pm},f}\alpha=0$ and that $d_{A^{\pm}}^{\ast}\alpha-\tau\ast[\ast F_{A^{\pm}}\wedge d_{A^{\pm}}\alpha]=0$. Because the perturbed Hessian $H_{A^{\pm},f}$ is nondegenerate it follows that $\alpha=d_{A^{\pm}}\ph$ for some $\ph\in\Omega^0(\Sigma,\ad(P))$. Then the second of these equations implies that
\begin{align*}
0&=d_{A^{\pm}}^{\ast}d_{A^{\pm}}\ph-\tau\ast[\ast F_{A^{\pm}}\wedge d_{A^{\pm}}d_{A^{\pm}}\ph]\\
&=d_{A^{\pm}}^{\ast}d_{A^{\pm}}\ph-\tau\ast[\ast F_{A^{\pm}}\wedge[F_{A^{\pm}}\wedge\ph]],
\end{align*}
from which it follows as before, by taking the inner product with $\ph$ and using irreducibility of $A^{\pm}$, that $\ph=0$ and hence also $\alpha=0$. We now argue by continuity of the spectral flow of the operator family $\tau\mapsto C_{\tau,f}^{\pm}$ that the (finite) number of negative eigenvalues of $C_{\tau,f}^{\pm}$ does not change with $\tau$. The statement of the lemma now follows because the total number of eigenvalue crossings of the operator family $s\mapsto C_f$, $s\in\R$, equals the difference of the numbers of negative eigenvalues of $C_f^-$ and $C_f^+$. By the preceding argumentation this difference is equal to the difference between the Morse indices of the augmented Yang--Mills Hessians at $A^-$, respectively $A^+$, whence the first claim. 
\end{proof}

\subsubsection*{Fredholm property and index in the case $1<p<\infty$}

\begin{proof}{\bf{[Theorem \ref{thm:Fredholmthm} in the case $1<p<\infty$].}}
The assertion of the theorem in the general case is a consequence of the following standard line of arguments (cf.~e.g.~\cite[Propositions 3.15--3.18]{Web2} for details). We fix a constant $\delta>0$ smaller than the modulus of any eigenvalue of the limiting operators $B_f^{\pm}$. Such a choice is possible since these operators are by assumption nondegenerate.

\setcounter{step}{0}
\begin{step}\label{step1}
Let $N$ be the vector space
\begin{multline*}
N=\{\xi=(\alpha,v,\psi)\in C^{\infty}(\R\times\Sigma)\mid\mathcal D_{(A,\omega,0)}\xi=0,\exists c>0\;\forall s\in\R\colon\\
\|\xi(s)\|_{L^{\infty}(\Sigma)}+\|\partial_s\xi(s)\|_{L^{\infty}(\Sigma)}+\|\nabla_A\xi(s)\|_{L^{\infty}(\Sigma)}\leq ce^{-\delta|s|}\}.
\end{multline*}
Then $\ker(\mathcal D_{(A,\omega,0)}\colon\mathcal W^p\to\mathcal L^p)=N$. In particular, this kernel is finite-dimensional and does not depend on $p$.
\end{step}

By standard elliptic estimates, any solution of $\mathcal D_{(A,\omega,0)}\xi=0$ is smooth. Exponential decay in $L^2(\Sigma)$ of any such $\xi$ holds by \eqref{expcov:ineq_f}. Elliptic bootstrapping arguments then show exponential decay in the form stated above. This proves the inclusion $\ker(\mathcal D_{(A,\omega,0)}\colon\mathcal W^p\to\mathcal L^p)\subseteq N$. The opposite inclusion is clearly satisfied.

\begin{step}\label{step2}
Let $N^{\ast}$ be the vector space
\begin{multline*}
N^{\ast}=\{\eta\in C^{\infty}(\R\times\Sigma)\mid\mathcal D_{(A,\omega,0)}^{\ast}\eta=0,\exists c>0\;\forall s\in\R\colon\\
\|\eta(s)\|_{L^{\infty}(\Sigma)}+\|\partial_s\eta(s)\|_{L^{\infty}(\Sigma)}+\|\nabla_A\eta(s)\|_{L^{\infty}(\Sigma)}\leq ce^{-\delta|s|}\}.
\end{multline*}
Then $\coker(\mathcal D_{(A,\omega,0)}^{\ast}\colon\mathcal W^p\to\mathcal L^p)=N^{\ast}$. In particular, this cokernel is finite-dimensional and does not depend on $p$.
\end{step}

The statement follows from Step \ref{step1} using the reflection $s\mapsto-s$ under which $\mathcal D_{(A,\omega,0)}^{\ast}$ is mapped to $-\mathcal D_{(A,\omega,0)}$.

\begin{step}\label{step3}
The range of the operator $\mathcal D_{(A,\omega,0)}\colon\mathcal W^p\to\mathcal L^p$ is closed.
\end{step}

First, there exists a constant $c=c(A,\omega,p)>0$ such that for any $s_0\in\R$ and the two intervals $I_1=(s_0-1,s_0+1)$ and $I_2=(s_0-2,s_0+2)$ the local elliptic estimate 
\begin{eqnarray*}
\|\xi\|_{W^{1,p}(I_1\times\Sigma)}\leq c(\|\mathcal D_{(A,\omega,0)}\xi\|_{L^p(I_2\times\Sigma)}+\|\xi\|_{L^p(I_2\times\Sigma)})
\end{eqnarray*}
is satisfied for all $\xi\in W^{1,p}(I\times\Sigma)$. Second, for the two $s$-independent solutions $(A^{\pm},\omega^{\pm},0)$ of \eqref{pertEYM1} the operator $\mathcal D_{(A^{\pm},\omega^{\pm},0)}\colon W^{1,p}(Z^{\pm})\to L^p(Z^{\pm})$ has a bounded inverse. Here we let $Z^{\pm}$ denote the half-infinite cylinders $Z^-=(-\infty,0)\times\Sigma$, respectively $Z^+=(0,\infty)\times\Sigma$. By a standard cut-off function argument, both estimates can be combined into the further estimate 
\begin{eqnarray*}
\|\xi\|_{\mathcal W^p}\leq c(\|\mathcal D_{(A,\omega,0)}\xi\|_{\mathcal L^p}+\|K\xi\|_{\mathcal L^p}),
\end{eqnarray*}
where $K\colon\mathcal W^p\to\mathcal L^p$ is a suitable compact operator. The assertion now follows from the abstract closed range lemma, cf.~\cite[p.~14]{Sal2}.

\begin{step}
We prove the theorem.
\end{step}

By the preceding steps, the operator $\mathcal D_{(A,\omega,0)}\colon\mathcal W^p\to\mathcal L^p$ is a Fredholm operator, its index being independent of $p$. Hence the asserted index formula \eqref{eq:Fredholmindex} follows from the one already established in the case $p=2$.
\end{proof}

\section{Compactness}\label{sect:compactness}

In this section we prove a compactness theorem for solutions of \eqref{pertEYM1} of uniformly bounded energy. Let $I\subseteq\R$ be an interval. For convenience we shall identify at several instances a path $(A,\Psi)\in C^{\infty}(I,\A(P)\times\Omega^0(\Sigma,\ad(P)))$ with the connection $\AA:=A+\Psi\,ds\in\mathcal A(I\times P)$. Its curvature is 
\begin{eqnarray}\label{eq:curvthreedim}
F_{\AA}=F_A+(\partial_s A+d_A\Psi)\wedge ds\in\Omega^2(I\times\Sigma,\ad(I\times P)).
\end{eqnarray}
We call a connection $\AA_1\in\mathcal A(I\times P)$ to be in {\bf{local slice with respect to the reference connection}} $\AA$ if the difference $\AA_1-\AA=\alpha+\psi\,ds$ satisfies
\begin{eqnarray*}
d_{\AA}^{\ast}(\alpha+\psi\,ds)=0.
\end{eqnarray*}
This condition is equivalent to  
\begin{eqnarray}\label{eq:locslicecond1}
\nabla_s\psi-d_A^{\ast}\alpha=0,
\end{eqnarray}
where we denote $\nabla_s\psi:=\partial_s\psi+[\Psi,\psi]$.\\
\noindent\\
In the following we fix $(A_0,\omega_0,\Psi_0)\in\A(P)\times\Omega^0(\Sigma,\ad(P))\times\Omega^0(\Sigma,\ad(P))$ as a smooth reference point and denote $\AA_0:=A_0+\Psi_0\,ds$. Thus $\AA_0$ is a connection on $I\times\Sigma$ whose components do not depend on the time-parameter $s$. Let $(A,\omega,\Psi)=(A_0,\omega_0,\Psi_0)+\xi$ with $\xi=(\alpha,v,\psi)$ be a smooth solution of \eqref{pertEYM1} on $I\times\Sigma$. We augment \eqref{pertEYM1} with the local slice condition \eqref{eq:locslicecond1} relatively to the reference connection $\AA_0$. Expanding the thus obtained system of equations about $(A_0,\omega_0,\Psi_0)$ yields the equation  
\begin{align}\label{eq:ExpansionEYM}
0=&\mathcal F(A_0,\omega_0,\Psi_0)+(\nabla_s+B_{(A_0,\omega_0,\Psi_0)})\xi+Q_{\omega_0}\xi\\
\nonumber&+(-X_f(A,\omega),Y_f(A,\omega),0),
\end{align}
where the linear operator $B_{(A_0,\omega_0,\Psi_0)}$ has been defined in \eqref{linoperator} and where we set
\begin{eqnarray*}
\mathcal F(A_0,\omega_0,\Psi_0):=\begin{pmatrix}\dot A_0-d_{A_0}\Psi_0-\ast d_{A_0}\omega_0\\\dot\omega_0+[\Psi_0,\omega_0]-\omega_0+\ast F_{A_0}\\0\end{pmatrix}
\end{eqnarray*}
and
\begin{eqnarray*}
Q_{\omega_0}\begin{pmatrix}\alpha\\v\\\psi\end{pmatrix}:=\begin{pmatrix}-[\alpha\wedge\psi]-\ast[\alpha\wedge v]\\ [\psi,v]+ \frac{1}{2}\ast[\alpha\wedge\alpha] \\ -\ast[\omega_0\wedge\ast v]\end{pmatrix}.
\end{eqnarray*}
We furthermore define the {\bf{gauge-invariant energy density}} of the solution $(A,\omega,\Psi)$ of \eqref{pertEYM1} to be the function
\begin{multline}\label{eq:endensity}
e(A,\omega,\Psi)\coloneqq\frac{1}{2}(|\ast d_A\omega+d_A\Psi+X_f(A,\omega)|^2\\
+|\ast F_A-\omega+[\Psi,\omega]+Y_f(A,\omega)|^2)\colon I\times\Sigma\to\R.
\end{multline}

\begin{prop}\label{prop:locunifconvergence}
Let $I\subseteq\R$ be a compact interval and $(A^{\nu},\omega^{\nu},0)$ be a sequence of smooth solutions of Eq.~\eqref{pertEYM1} on $I\times\Sigma$ in temporal gauge. Assume that there exists a constant $C>0$ such that for $e^{\nu}:=e(A^{\nu},\omega^{\nu},0)$
\begin{eqnarray}\label{eq:unifenergydensity}
\|e^{\nu}\|_{L^1(I\times\Sigma)}\leq C
\end{eqnarray}
for all $\nu\in\N$. Then there exists a subsequence, still denoted by $(A^{\nu},\omega^{\nu},0)$, and a sequence of gauge transformations $g^{\nu}\in\G(I\times P)$ such that the sequence $(g^{\nu})^{\ast}(A^{\nu},\omega^{\nu},0)$ converges in the $C^k$ topology, for every $k\in\mathbbm N_0$.
\end{prop}

\begin{proof}
Combining \eqref{pertEYM1} and \eqref{eq:curvthreedim} it follows that the curvature of the connection $\AA^{\nu}:=A^{\nu}+0\,ds\in\A(I\times P)$ is
\begin{eqnarray*}
F_{\AA^{\nu}}=F_{A^{\nu}}+(\ast d_{A^{\nu}}\omega^{\nu}+X_f(A^{\nu},\omega^{\nu}))\wedge ds.
\end{eqnarray*}
Assumption \eqref{eq:unifenergydensity} together with the definition \eqref{eq:endensity} of $e^{\nu}$ implies the uniform curvature bound
\begin{multline*}
\|F_{\AA^{\nu}}\|_{L^2(I\times\Sigma)}\leq\|\ast F_{A^{\nu}}-\omega^{\nu}+Y_f(A^{\nu},\omega^{\nu})\|_{L^2(I\times\Sigma)}+\|\omega^{\nu}\|_{L^2(I\times\Sigma)}\\
+\|Y_f(A^{\nu},\omega^{\nu})\|_{L^2(I\times\Sigma)}+\|\ast d_{A^{\nu}}\omega^{\nu}+X_f(A^{\nu},\omega^{\nu})\|_{L^2(I\times\Sigma)}\leq C_1
\end{multline*}
for some further constant $C_1$. The term $\|\omega^{\nu}\|_{L^2(I\times\Sigma)}$ is uniformly bounded as follows from Lemmata \ref{lem:L2boundFA} and \ref{lem:L2estomega}. A standard estimate allows to uniformly bound the $L^2$-norm of $Y_f(A^{\nu},\omega^{\nu})$, cf.~\cite[Proposition D.1 (iii)]{SalWeh}. Therefore the assumptions of Uhlenbeck's weak compactness theorem (cf.~\cite[Theorem A]{Wehrheim}) are satisfied for the Sobolev exponent $p=2$. It yields the existence of a sequence of gauge transformations $g^{\nu}\in\G^{2,p}(I\times P)$ such that after passing to a subsequence
\begin{eqnarray*}
(g^{\nu})^{\ast}\AA^{\nu}\rightharpoonup\AA^{\ast}=A^{\ast}+\Psi^{\ast}\,ds\qquad(\nu\to\infty)
\end{eqnarray*}
weakly in $W^{1,2}(I\times\Sigma)$ for some limiting connection $\AA^{\ast}\in W^{1,2}(I\times\Sigma)$. Let $g_0\in\G^{2,2}(I\times\Sigma)$ be a gauge transformation such that $\tilde\AA_0:=g_0^{\ast}\AA^{\ast}$ is in temporal gauge, i.e.~of the form $\tilde\AA_0=\tilde A_0+0\,ds$. Let $\AA_0=A_0+0\,ds$ be a smooth connection $W^{1,2}$ close to $\tilde\AA_0$. Because the weakly convergent sequence $(g^{\nu})^{\ast}\AA^{\nu}$ is bounded in $W^{1,2}$, this is also the case for the gauge transformed sequence $(g^{\nu}g_0)^{\ast}\AA^{\nu}$. Now the local slice theorem (cf.~\cite[Theorem F]{Wehrheim}) yields a further sequence $h^{\nu}$ of gauge transformations such that $(g^{\nu}g_0h^{\nu})^{\ast}\AA^{\nu}$ is a bounded sequence in $W^{1,2}$, which in addition is in local slice with respect to the reference connection $\AA_0$. Denoting $\alpha^{\nu}+\psi^{\nu}\,ds:=(g^{\nu}g_0h^{\nu})^{\ast}\AA^{\nu}-\AA_0$ this means that
\begin{eqnarray}\label{eq:locslicecond}
d_{\AA_0}^{\ast}(\alpha^{\nu}+\psi^{\nu}\,ds)=d_{A_0}^{\ast}\alpha^{\nu}-\dot\psi^{\nu}=0
\end{eqnarray}
for all $\nu\in\N$. Differentiation of \eqref{pertEYM1} with respect to $s$ shows that each $\omega^{\nu}$ satisfies
\begin{eqnarray*}
0=-\ddot\omega^{\nu}+\Delta_{A^{\nu}}\omega^{\nu}+\dot\omega^{\nu}+d_{A^{\nu}}X_f(A^{\nu},\omega^{\nu})-\partial_sY_f(A^{\nu},\omega^{\nu}).
\end{eqnarray*}
The uniform $L^2$-bound satisfied by $\omega^{\nu}$ together with ellipticity of the linear operator $-\frac{d^2}{ds^2}+\frac{d}{ds}+\Delta_{A^{\nu}}$ shows that the sequence $\omega^{\nu}$ is in fact uniformly bounded in $W^{2,2}$. This uses in addition that the terms $d_{A^{\nu}}X_f(A^{\nu},\omega^{\nu})$ and $\partial_sY_f(A^{\nu},\omega^{\nu})$ are uniformly bounded in $L^2$ which follows from the estimates on holonomy perturbations in \cite[Proposition D.1 (iv,vi)]{SalWeh}. We now choose $\omega_0\in\Omega^0(\Sigma,\ad(P))$ for reference and set $v^{\nu}:=\omega^{\nu}-\omega_0$. By \eqref{eq:ExpansionEYM} and \eqref{eq:locslicecond} each $\xi^{\nu}$ satisfies the equation
\begin{eqnarray*} 
(\nabla_s+B_{(A_0,\omega_0,0)})\xi^{\nu}=-\mathcal F(A_0,\omega_0,0)-Q_{\omega_0}\xi^{\nu}+(X_f(A^{\nu},\omega^{\nu}),-Y_f(A^{\nu},\omega^{\nu}),0).
\end{eqnarray*}
The last three terms on the right-hand side of this equation are uniformly bounded in $W^{1,\frac{3}{2}}(I\times\Sigma)$. Namely, the map $Q_{\omega_0}\colon W^{1,2}(I\times\Sigma)\to W^{1,\frac{3}{2}}(I\times\Sigma)$ maps bounded sets to bounded sets as follows from Sobolev multiplication and embedding theorems, and $\xi^{\nu}$ satisfies a uniform bound in $W^{1,2}(I\times\Sigma)$. The same holds true for the term $(X_f(A^{\nu},\omega^{\nu}),-Y_f(A^{\nu},\omega^{\nu}),0)$ by a standard estimate, cf.~\cite[Proposition D.1 (vi)]{SalWeh}. The statement of the proposition now follows inductively from a bootstrap argument based on ellipticity of the linear operator $\nabla_s+B_{(A_0,\omega_0,0)}$.
\end{proof}

The following theorem states compactness of moduli spaces up to convergence to broken trajectories. The notion of $a$-regular perturbations is introduced in Definition \ref{def:regularorbit} below. We refer to \textsection \ref{subsect:pertcritpt} for explanation of some of the subsequently used notations.

\begin{thm}\label{thm:compactness}
Fix a number $a>0$ and let $f$ be an $a$-regular perturbation. Let $A_-^{\nu}$ and $A_+^{\nu}$ be two sequences in $\crit(\J+h_f)$ which converge uniformly to some critical point $A^-$, respectively $A^+$ of $\J+h_f$. Let $(A^{\nu},\omega^{\nu},0)$ be a sequence in $\widehat{\mathcal M}_f(A_-^{\nu},A_+^{\nu})$ with bounded energy 
\begin{eqnarray}\label{eq:boundedenergyass}
\sup_{\nu}(\mathcal J(A^{\nu},\omega^{\nu},0)+h_f(A^{\nu},\omega^{\nu}))<a.
\end{eqnarray}
Then there is a subsequence, still denoted by $(A^{\nu},\omega^{\nu},0)$, finitely many critical points $B_0=A^-,\ldots,B_{\ell}=A^+\in\crit^a(\mathcal J+h_f)$ and connecting trajectories $(A_i,\omega_i,\Psi_i)\in\widehat{\mathcal M}_f(B_i,B_{i+1})$ for $i=0,\ldots,\ell-1$, such that $(A^{\nu},\omega^{\nu},0)$ converges to the broken trajectory 
\begin{eqnarray*}
((A_0,\omega_0,\Psi_0),\ldots,(A_{\ell-1},\omega_{\ell-1},\Psi_{\ell-1}))
\end{eqnarray*}
in the following sense. For every $i=0,\ldots,\ell-1$ there is a sequences $s_i^{\nu}\in\R$ and a sequence of gauge transformations $g_i^{\nu}\in\mathcal G(\R\times\Sigma)$ such that the $s$-dilated sequence $(g_i^{\nu})^{\ast}(A^{\nu}(\,\cdot\,+s_i^{\nu}),\omega^{\nu}(\,\cdot\,+s_i^{\nu}),0)$ converges uniformly to $(A_i,\omega_i,\Psi_i)$ on compact subsets of $\R\times\Sigma$.  
\end{thm}

\begin{proof}
We first note that our assumptions imply for every compact interval $I\subseteq\R$ the existence of a constant $C>0$ such that uniform energy bound \eqref{eq:unifenergydensity} is satisfied. This follows for $I=\R$ (and hence for every subinterval $I$) from \eqref{eq:boundedenergyass} and the identity
\begin{multline*}
\int_{-\infty}^{\infty}\|e^{\nu}(s)\|_{L^2(\Sigma)}^2\,ds=\\\J(A_-^{\nu},\omega_-^{\nu},0)-\J(A_+^{\nu},\omega_+^{\nu},0)+h_f(A_-^{\nu},\omega_-^{\nu})-h_f(A_+^{\nu},\omega_+^{\nu}).
\end{multline*}
Hence the restriction of $(A^{\nu},\omega^{\nu},0)$ to each compact interval $I\subseteq\R$ satisfies the assumptions of Proposition \ref{prop:locunifconvergence}. Therefore, after passing to a subsequence and modification by gauge transformations it follows that $(A^{\nu}|_I,\omega^{\nu}|_I,0)$ converges uniformly to a limit $(A^{\ast}|_I,\omega^{\ast}|_I,\Psi^{\ast}|_I)$, which again satisfies Eq.~\eqref{pertEYM1} on $I\times\Sigma$. The remaining parts of the statement now follow from standard arguments, which e.g.~can be found in \cite[\textsection 7]{SalWeh}. 
\end{proof}

\section{Exponential decay}\label{sect:exponentialdecay}

Throughout we fix a constant $a>0$.

\begin{prop}[Stability of injectivity]\label{expconv:lemmaest}
Let $h_f$ be an $a$-regular perturbation in the sense of Definition \ref{def:regularorbit}. Then there are positive constants $\tilde\delta$ and $c$ such that the following holds. Let $(A,\omega) \in \A(P)\times \Omega^0(\Sigma,\ad(P))$ with $\J(A,\omega)<a$ satisfy 
\begin{equation}\label{expconv:lemmaest:linf}
\|\ast d_A\omega+X_f(A,\omega)\|_{L^\infty(\Sigma)}+\|*F_A-\omega+Y_f(A,\omega)\|_{L^\infty(\Sigma)}\leq \tilde\delta.
\end{equation}
Then for every $\xi=(\alpha,v,\psi)$ of class $C^1$, where $\alpha\in\Omega^1(\Sigma,\ad(P))$ and $v,\psi\in\Omega^0(\Sigma,\ad(P))$, it holds that 
\begin{eqnarray}\label{expconv:lemmaest:linf2}
 \|\xi\|_{L^2(\Sigma)}^2=\|\alpha\|_{L^2(\Sigma)}^2+\|v\|_{L^2(\Sigma)}^2+\|\psi\|_{L^2(\Sigma)}^2\\
\nonumber\leq c\big (\left\|*d_Av+d_A\psi-*[\omega\wedge\alpha]+\d\!X_f(A,\omega)(\alpha,v)\right\|^2_{L^2(\Sigma)}\\
\nonumber+\left\|*d_A\alpha-v-[\omega\wedge \psi]+\d\!Y_f(A,\omega)(\alpha,v)\right\|^2_{L^2(\Sigma)}+\left\|d_A^*\alpha-*[\omega\wedge * v]\right\|^2_{L^2(\Sigma)}\big).
\end{eqnarray}
(This estimate shows stable invertibility of the Hessian of the perturbed functional $\mathcal J+h_f$ near critical points, cf.~\eqref{linoperator}).
\end{prop}

\begin{proof}
We assume by contradiction that there is  $(A,\omega)$ as in the statement of the proposition but \eqref{expconv:lemmaest:linf2} fails to be true for every $c>0$. Then there is a sequence $(A_\nu,\omega_\nu)\in \A(P)\times \Omega^0(\Sigma,\ad(P))$ and a nullsequence $\delta_\nu\to 0$ such that 
\begin{equation}\label{eq:nullsequence}
\|\ast d_{A_\nu}\omega_\nu+X_f(A^{\nu},\omega^{\nu})\|_{L^\infty(\Sigma)}+\|*F_{A_\nu}-\omega_\nu+Y_f(A^{\nu},\omega^{\nu})\|_{L^\infty(\Sigma)}\leq \delta_\nu
\end{equation}
and inequality \eqref{expconv:lemmaest:linf2} does not hold for the constant $c_\nu=\nu$. For any $(A,\omega)$ we obtain from the definition of $\J$ and the Cauchy--Schwartz inequality that
\begin{multline*}
\frac{1}{2} \|\omega\|_{L^2(\Sigma)}^2= \J(A,\omega)-\int_{\Sigma}\langle*\omega \wedge (F_A-*\omega)\rangle\\
\leq  a+ \frac 14 \|\omega\|_{L^2(\Sigma)}^2+\|*\omega-F_A\|_{L^2(\Sigma)}^2.
\end{multline*}
This inequality together with  \eqref{expconv:lemmaest:linf} shows that $\omega_\nu$ and $F_{A_\nu}$ are uniformly bounded in the $L^2$-norm. Again by  \eqref{expconv:lemmaest:linf}, $d_{A_{\nu}}\omega_{\nu}=\nabla_{A_{\nu}}\omega_{\nu}$ is uniformly bounded in $L^{\infty}$. Hence altogether, the sequence $\omega_{\nu}$ is uniformly bounded in $W^{1,\infty}$. Furthermore, \eqref{expconv:lemmaest:linf} shows that $F_{A_\nu}$ is uniformly bounded in $L^{\infty}$ and thus by Uhlenbeck's weak compactness theorem (cf.~\cite[Theorem A]{Wehrheim}) we can assume that (after modifying by suitable gauge transformations) the sequence $(A_\nu)$ has a weakly convergent subsequence in $\A^{1,p}(P)$ for any fixed $2<p<\infty$. Therefore, there is a subsequence of $(A_\nu,\omega_\nu)$ which converges weakly in $W^{1,p}$ and strongly in $C^0$ to a limit $(A_{\ast},\omega_{\ast})$. It follows from \eqref{eq:nullsequence} that $(A_{\ast},\omega_{\ast})$ is a critical point of $\J+h_f$. Also, $\J(A_{\ast},\omega_{\ast})\leq a$. Hence \eqref{expconv:lemmaest:linf2} holds for $(A_{\ast},\omega_{\ast})$ and some constant $c>0$, because the critical points below level $a$ are not degenerate. The existence of such a finite constant $c$ contradicts our above assumption. The assertion now follows.
\end{proof}

\begin{thm}\label{thm:L2expconv}
Let $h_f$ be an $a$-regular perturbation in the sense of Definition \ref{def:regularorbit}. Then there exist positive constants $\delta$, $\tilde\delta$, and $c$ such that the following holds. Let $(A,\omega,\Psi)$ satisfy \eqref{pertEYM1} on $\R\times\Sigma$  such that $\J(A(s),\omega(s))<a$ for all $s\in\R$. Assume that there exists $s_0>0$ such that $(A,\omega,\Psi)$ satisfies 
\begin{equation}\label{expconv:linf}
\|\partial_sA(s)-d_{A(s)}\Psi(s)\|_{L^\infty(\Sigma)}+\|\nabla_s\omega(s)\|_{L^\infty(\Sigma)}\leq \tilde\delta
\end{equation}
for all $|s|>s_0$. Then for every $\xi=(\alpha,v,\psi)$ of class $C^2$, where $\alpha(s)\in\Omega^1(\Sigma,\ad(P))$ and $v(s),\psi(s)\in\Omega^0(\Sigma,\ad(P))$, which satisfies $\mathcal D_{(A,\omega,\Psi)}\xi=0$ and does not diverge as $s\to\pm\infty$, it follows that
\begin{equation} \label{expcov:ineq_f}
\|\xi(s)\|_{L^2(\Sigma)}\leq c e^{-\delta |s|}
\end{equation}
for all $s>s_0$.
\end{thm}

\begin{proof}
For fixed $\xi=(\alpha,v,\psi)$ as in the assumptions let us consider the function
\begin{eqnarray*}
f\colon\R\to\R,\qquad s\mapsto\frac{1}{2}\Big(\|\alpha(s)\|_{L^2(\Sigma)}^2+\|v(s)\|_{L^2(\Sigma)}^2+\|\psi(s)\|_{L^2(\Sigma)}^2\Big).
\end{eqnarray*}
We claim that $f$ satisfies the differential inequality
\begin{equation}\label{flow:proof:eqfdsjapo}
f''(s)\geq \delta^2 f(s)
\end{equation}
for a constant $\delta>0$ and all $s\geq 1$. Assuming this claim for a moment it implies that $f$ decays exponentially as $s\to\pm\infty$. Namely, since for $s\geq1$
\begin{equation*}
\frac{d}{ds}\left( e^{-\delta s}\left(f'(s)+\delta f(s)\right)\right)= e^{-\delta s}\left(f''(s)-\delta^2 f(s)\right) \geq 0,
\end{equation*} 
it follows that $f'(s)+\delta f(s)<0$. Otherwise the function $s\mapsto e^{-\delta s}\left(f'(s)+\delta f(s)\right)$ would for some $s_0\geq1$ be nonnegative. Since it is increasing it would then be nonnegative for all $s\geq s_0$. Thus, since $f(s)$ is by assumption bounded, $s\mapsto e^{-\delta s }f(s)$ would decrease and hence $f'(s)$ increase for $s\geq s_0$ sufficiently large. Therefore $f(s)$ would be unbounded which is a contradiction and hence $f'(s)+\delta f(s)<0$. Therefore, if the function $f$ satisfies \eqref{flow:proof:eqfdsjapo}, then
\begin{equation*}
f(s)\leq c_1e^{-\delta s }
\end{equation*}
for a suitable constant $c_1>0$. We now prove \eqref{flow:proof:eqfdsjapo}. Differentiating $f$ twice with respect to $s$, we obtain
\begin{equation}\label{eq:expconv1}
\begin{split}
f''=&\|\nabla_s\alpha\|_{L^2(\Sigma)}^2+\|\nabla_sv\|_{L^2(\Sigma)}^2+\|\nabla_s\psi\|_{L^2(\Sigma)}^2+ \langle\alpha,\nabla_s^2\alpha\rangle+\langle v,\nabla_s^2 v\rangle+\langle\psi,\nabla_s^2\psi\rangle \\
\geq&\|\nabla_s\alpha\|_{L^2(\Sigma)}^2+\|\nabla_sv\|_{L^2(\Sigma)}^2+\|\nabla_s\psi\|_{L^2(\Sigma)}^2+\frac{1}{2}\|d_A^{\ast}\alpha-[\omega\wedge v]\|_{L^2(\Sigma)}^2\\
&+\frac{1}{2}\|\ast d_A\alpha-v-[\omega\wedge\psi]+\d\!Y\|_{L^2(\Sigma)}^2+\frac{1}{2}\|\ast d_Av+d_A\psi+\d\!X-\ast[\omega\wedge\alpha]\|_{L^2(\Sigma)}^2\\
&-c\tilde\delta\left(\|\alpha\|_{L^2(\Sigma)}^2+\|v\|_{L^2(\Sigma)}^2+\|\psi\|_{L^2(\Sigma)}^2\right),
\end{split}
\end{equation}
where the second equality follows from  the computation  below. To obtain the desired inequality \eqref{flow:proof:eqfdsjapo} we apply \eqref{expconv:lemmaest:linf2} after choosing $\tilde\delta>0$ still smaller if necessary. For abbreviation we  set
\begin{eqnarray*}
\d\!X\coloneqq\d\!X_f(A,\omega)(\alpha,v)\qquad\textrm{and}\qquad\d\!\dot X=\nabla_s(\d\!X)-\d\!X_f(A,\omega)(\dot\alpha,\dot v)
\end{eqnarray*}
(and likewise for $\d\!Y_f(A,\omega)(\alpha,v)$) in the following calculation. Using at several places the assumption that $\mathcal D_{(A,\omega,\Psi)}\xi=0$, we compute
\begin{equation*}\label{eq:expconv2}
\begin{split}
\langle\alpha,\nabla_s^2\alpha\rangle&+\langle v,\nabla_s^2 v\rangle+\langle\psi,\nabla_s^2\psi\rangle\\
=&\langle \alpha,\nabla_s\left(-*[\omega\wedge\alpha]+\ast d_Av+d_A\psi+\d\!X \right)\rangle\\
&+\langle v, \nabla_s \left(-*d_A\alpha+v+[\omega\wedge\psi]-\d\!Y \right) \rangle
+\langle \psi,\nabla_s\left( d_A^*\alpha-\ast[\omega\wedge *v]\right)\rangle\\
=&\langle\alpha,\ast d_A\nabla_sv+[\nabla_s,\ast d_A]v\rangle+\langle \alpha, d_A\nabla_s\psi+[\nabla_s,d_A]\psi\rangle\\
&-\langle\alpha,*[\nabla_s\omega\wedge\alpha]+*[\omega\wedge\nabla_s\alpha]\rangle
+\langle\alpha,\nabla_s\d\!X \rangle-\langle v,*d_A\nabla_s\alpha+[\nabla_s,*d_A]\alpha\rangle\\
&+\langle v,\nabla_s v\rangle+\langle v , [\nabla_s\omega\wedge\psi]+[\omega\wedge\nabla_s\psi]\rangle+\langle v,\nabla_s\d\!Y \rangle\\
&+\langle\psi,d_A^*\nabla_s\alpha+[\nabla_s,d_A^*]\alpha\rangle
-\langle\psi,*[\nabla_s\omega\wedge*v]+*[\omega\wedge*\nabla_sv]\rangle\\
=&\langle\ast d_A\alpha,\ast d_A\alpha-v-[\omega\wedge\psi]+\d\!Y\rangle+\langle\alpha,[\nabla_s,\ast d_A]v\rangle\\
&+\langle d_A^{\ast}\alpha,d_A^{\ast}\alpha-[\omega\wedge v]\rangle+\langle\alpha,[\nabla_s,d_A]\psi\rangle\\
&-\langle\alpha,\ast[\nabla_s\omega\wedge\alpha]+\ast[\omega\wedge\nabla_s\alpha]\rangle+\langle\alpha,\nabla_s\d\!X\rangle\\
&+\langle\ast d_Av,\ast d_Av-\ast[\omega\wedge\alpha]+d_A\psi+\d\!X\rangle-\langle v,[\nabla_s,\ast d_A]\alpha\rangle\\
&+\langle v,\nabla_s v\rangle+\langle v,[\nabla_s\omega\wedge\psi]+[\omega\wedge\nabla_s\psi]\rangle+\langle v,\nabla_s \d\!Y\rangle\\
&+\langle d_A\psi,d_A\psi+\ast d_Av+\d\!X-\ast[\omega\wedge\alpha]\rangle-\langle\psi,[\nabla_s,d_A^*]\alpha\rangle\\
&-\langle\psi,\ast[\nabla_s\omega\wedge\ast v]+\ast[\omega\wedge\ast\nabla_s v]\rangle\\
=&\|\ast d_Av+d_A\psi-\ast[\omega\wedge\alpha]+\d\!X\|^2+\|\ast d_A\alpha-v-[\omega\wedge\psi]+\d\!Y\|^2\\
&+\|d_A^{\ast}\alpha-[\omega\wedge v]\|^2+\langle\alpha,\dot X\rangle+\langle v,\dot Y\rangle\\
&+\langle\alpha,[\nabla_s,\ast d_A]v\rangle-\langle\psi,[\nabla_s,d_A^*]\alpha\rangle+\langle\alpha,[\nabla_s,d_A]\psi\rangle-\langle v,[\nabla_s,\ast d_A]\alpha\rangle\\
&-\langle\alpha,\ast[\nabla_s\omega\wedge\alpha]\rangle-\langle\psi,\ast[\nabla_s\omega\wedge\ast v]\rangle+\langle v,[\nabla_s\omega\wedge\psi]\rangle
\end{split}
\end{equation*}
Note that the terms appearing in the last three lines are bilinear expressions in $\alpha$, $v$ or $\psi$, which with the help of \eqref{expconv:linf} can be bounded through
\begin{eqnarray*}
c\tilde\delta\left(\|\alpha\|_{L^2(\Sigma)}^2+\|v\|_{L^2(\Sigma)}^2+\|\psi\|_{L^2(\Sigma)}^2\right)
\end{eqnarray*}
for some universal constant $c$. This completes the proof of inequality \eqref{eq:expconv1}. Thus, using Proposition \ref{expconv:lemmaest} and choosing $\tilde \delta$  still smaller if necessary we conclude that
\begin{equation}
\begin{split}
f''(s)\geq\delta^2\left(\|\alpha\|_{L^2(\Sigma)}^2+\|v\|_{L^2(\Sigma)}^2+\|\psi\|_{L^2(\Sigma)}^2\right)=\delta^2 f(s),
\end{split}
\end{equation}
and thus the claimed exponential convergence follows.  
\end{proof}

We can now state and prove the main result of this section.

\begin{thm}[Exponential decay]\label{thm:mainexpconvergence}
For a constant $a>0$, let $h_f$ be an $a$-regular perturbation of $\J$. Let $(A,\omega,\Psi)$ be a smooth solution of \eqref{pertEYM1} on $\R\times\Sigma$  such that $\limsup_{s\in\R}\J(A(s),\omega(s),\Psi(s))<a$ and suppose that
\begin{eqnarray}\label{eq:finiteenergycond}
\int_{\R}\int_{\Sigma}|\partial_sA(s,z),\partial_s\omega(s,z),\partial_s\Psi(s,z)|^2\,\operatorname{dvol}_{\Sigma}ds<\infty.
\end{eqnarray}
Then there are critical points $(A^{\pm},\omega^{\pm})$ of $\J+h_f$ such that $(A(s),\omega(s),\Psi(s))$ converges to $(A^{\pm},\omega^{\pm},0)$ as $s\to\pm\infty$. Moreover, there are constants $\delta>0$ and $C_k$, $k\in\N_0$, such that
\begin{eqnarray*}
\|(A^{\pm}-A(s),\omega^{\pm}-\omega(s))\|_{C^k([s-1,s+1]\times\Sigma)}\leq C_ke^{-\delta|s|}
\end{eqnarray*}
for every $|s|\geq1$ and $k\in\N_0$.
\end{thm}

\begin{proof}
We first verify that $(A,\omega,\Psi)$ meets the assumptions of Theorem \ref{thm:L2expconv}. Let some constant $\tilde\delta$ be given. Then \eqref{eq:finiteenergycond} implies that for any $\delta>0$ there exists $s_0>0$ such that for all $|s|\geq s_0$
\begin{eqnarray*}
\|e(A,\omega,\Psi)\|_{L^1([s-1,s+1]\times\Sigma)}=\frac{1}{2}\|(\partial_sA,\partial_s\omega,\partial_s\Psi)\|_{L^2([s-1,s+1]\times\Sigma)}^2\leq\delta,
\end{eqnarray*}
where $e(A,\omega,\Psi)$ denotes the gauge-invariant energy density as defined in \eqref{eq:endensity}. Lemma \ref{lem:pointwiseenergydesity} now yields the estimate
\begin{eqnarray*}
\|e(A(s),\omega(s),\Psi(s))\|_{L^{\infty}(\Sigma)}\leq C\delta,
\end{eqnarray*}
for every $|s|\geq s_0$ and some constant  $C>0$ which only depends on $\Sigma$. Therefore assumption \eqref{expconv:linf} of Theorem \ref{thm:L2expconv} is satisfied. We apply this theorem to $\xi=(\partial_sA,\partial_s\omega,\partial_s\Psi)$. This yields existence of the integral
\begin{eqnarray*}
(A^{\pm},\omega^{\pm},\Psi^{\pm})\coloneqq(A(s_0),\omega(s_0))+\int_{s_0}^{\pm\infty}(\partial_sA,\partial_s\omega,\partial_s\Psi)\,ds
\end{eqnarray*}
and exponential convergence in $L^2$ of $(A(s),\omega(s),\Psi(s))$ to $(A^{\pm},\omega^{\pm},\Psi^{\pm})$ as $s\to\pm\infty$. After applying a suitable time-dependent gauge transformation we may assume that $\Psi^{\pm}=0$. The claimed exponential convergence with respect to the $C^k$ norm for any $k\geq0$ may now be obtained by standard bootstrapping arguments. 
\end{proof}

\section{Transversality}\label{sect:transversality}

\subsection{Holonomy perturbations}\label{subsec:holonomypert}

To achieve transversality we introduce a perturbation scheme based on so-called holonomy perturbations. Such perturbations have been used in similar situations, cf.~\cite{Donaldson1,Floer, SalWeh,Taubes}. We adapt the construction of holonomy perturbations as carried out in \cite[Appendix D]{SalWeh} slightly in order to make it well suited in the situation at hand.\\
\noindent\\
In the following it will be convenient to identify the pair $(A,\omega)\in\A(P)\times\Omega^0(\Sigma,\ad(P))$ with the connection $\AA:=A+\omega\,dt$ on the principal $G$-bundle $P\times S^1$ over $\Sigma\times S^1$. We point out that both components $A$ and $\omega$ of $\AA$ do not depend on the second coordinate $t\in S^1$. We let $\tilde{\mathcal A}(P)\subseteq\mathcal A(P\times S^1)$ denote the set of such connections. The diagonal action of $\G(P)$ on $\A(P)\times\Omega^0(\Sigma,\ad(P))$ induces an action on $\tilde{\mathcal A}(P)$ given by $g^{\ast}\AA=g^{\ast}A+g^{-1}\omega g\,dt$. Let $\D\subseteq\C$ denote the closed unit disk. For an integer $m\geq1$ let $\Gamma_m$ denote the set of sequences $\gamma=(\gamma_1,\ldots,\gamma_m)$ of orientation preserving embeddings $\gamma_i\colon S^1\times\D\to \Sigma\times S^1$. Every $\gamma\in\Gamma_m$ gives rise to a map
\begin{eqnarray*}
\rho=(\rho_1,\ldots,\rho_m)\colon\D\times\tilde{\mathcal A}(P)\to G^{m}, 
\end{eqnarray*}
where $\rho_i(z,\AA)$ is the holonomy of the connection $\AA\in\tilde{\mathcal A}(P)$ along the loop $\theta\mapsto\gamma_i(\theta,z)$. Let $\mathcal F_m$ denote the space of real-valued smooth functions on $\D\times G^{m}$ invariant under the diagonal action by simultaneous conjugation of $G$ on $G^m$. Each pair $(\gamma,f)\in\Gamma_m\times\mathcal F_m$ determines a smooth function $h_f\colon\tilde{\mathcal A}(P)\to\R$ by
\begin{eqnarray}\label{eq:defholonpert}
h_f(\AA):=\int_{\D}f(z,\rho(z,\AA))\,dz.
\end{eqnarray}
The integrand is by construction invariant under the action of $\G(P)$ (in fact of the larger group $\G(P\times S^1)$) on the argument $\AA$. By the Riesz representation theorem, the differential $dh_f(\AA)(\alpha+v\,dt)$ in direction of $\alpha+v\,dt\in T_{\AA}\tilde{\mathcal A}(P)$ can be written in the form
\begin{eqnarray}\label{eq:gradienthf}
dh_f(\AA)(\alpha+v\,dt)=\int_{\Sigma}\langle-X_f(\AA)\wedge\ast\alpha\rangle+\langle Y_f(\AA)\wedge\ast v\rangle 
\end{eqnarray}
for uniquely defined $(X_f(\AA),Y_f(\AA))\in\Omega^1(\Sigma,\ad(P))\times\Omega^0(\Sigma,\ad(P))$. (For later convenience, we put a minus sign in front of $X_f(\AA)$). We define the {\bf{perturbed elliptic Yang--Mills flow}} to be the system of equations
\begin{eqnarray}\label{pertEYM}
\begin{cases}
0=\partial_sA-d_A\Psi-\ast d_A\omega-X_f(\AA),\\
0=\partial_s\omega+[\Psi,\omega]-\omega+\ast  F_A+Y_f(\AA).
\end{cases} 
\end{eqnarray}
Equations \eqref{pertEYM} are invariant under the action of time-dependent gauge transformations in $\G(P)$. Since in the following all statements are formulated in a gauge-invariant way we may always assume that $\Psi=0$.

\subsection{Perturbed critical point equation}\label{subsect:pertcritpt}

The above defined class of holonomy perturbations will turn out to be large enough to achieve surjectivity of the linearized operators $\mathcal D_{(A,\omega,\Psi)}$ along connecting trajectories $(A,\omega,\Psi)$. However, in order to achieve nondegeneracy of critical points of the perturbed functional $\J+h_f$ it is sufficient to work with the more restricted class of holonomy perturbations which depend on $A$ (but not on $\omega$). Such perturbations are maps $h_f\colon\A(P)\to\R$ defined as in \eqref{eq:defholonpert} where we set $\AA\coloneqq A+0\,dt$. In this case the (negative of the) gradient of $h_f$ at $A$ is the uniquely defined $X_f(A)\in\Omega^1(\Sigma,\ad(P))$ such that
\begin{eqnarray}\label{eq:gradienthf}
dh_f(A)(\alpha)=\int_{\Sigma}\langle-X_f(A)\wedge\ast\alpha\rangle 
\end{eqnarray}
for all $\alpha\in\Omega^1(\Sigma,\ad(P))$. Critical points of the thus perturbed elliptic Yang--Mills functional $\mathcal J+h_f$ are solutions $(A,\omega)$ of the system of equations
\begin{eqnarray}\label{eq:pertcritpointeq}
\begin{cases}0=\ast d_A\omega+X_f(A),\\
0=\ast  F_A-\omega.
\end{cases}
\end{eqnarray}
The set of critical points of the functional $\mathcal J+h_f$ is denoted by $\Crit(\mathcal J+h_f)$. For $a>0$ we also set
\begin{eqnarray*}
\crit^a(\mathcal J+h_f)\coloneqq\{(A,\omega)\in\crit(\J+h_f)\mid\J(A,\omega)<a\}.
\end{eqnarray*}
We define the linear operator $H_{(A,\omega)}^{\J,f}\colon W^{1,p}(\Sigma)\to L^p(\Sigma)$ to be the Hessian at $(A,\omega)$ of the perturbed functional $\J+h_f$, i.e.~
\begin{eqnarray*}
H_{(A,\omega)}^{\J,f}(\alpha,v)=(\ast d_Av-\ast[\omega\wedge\alpha]+\d\!X_f(A)\alpha,\ast d_A\alpha-v).
\end{eqnarray*}
Note that the previously defined operator $B_{(A,\omega,0)}$ (cf.~\eqref{linoperator}) is the Hessian $H_{(A,\omega)}^{\J,0}$ considered here, augmented by a gauge-fixing condition. Associated with each $(A,\omega)\in\Crit(\mathcal J+h_f)$ is the twisted de Rham complex
\begin{multline*}
\Omega^0(\Sigma,\ad(P))\overset{(d_A,[\omega\wedge\,\cdot\,])}{\longrightarrow}\Omega^1(\Sigma,\ad(P))\oplus\Omega^0(\Sigma,\ad(P))\\
\overset{H_{(A,\omega)}^{\J,f}}{\longrightarrow}\Omega^1(\Sigma,\ad(P))\oplus\Omega^0(\Sigma,\ad(P)).
\end{multline*}
The first operator in this complex is the infinitesimal action of the group $\G(P)$ at $(A,\omega)$. Let 
\begin{eqnarray*}
H_{(A,\omega)}^0\coloneqq\ker(d_A,[\omega\wedge\,\cdot\,]),\qquad H_{(A,\omega,f)}^1\coloneqq\frac{\ker H_{(A,\omega)}^{\J,f}}{\im(d_A,[\omega\wedge\,\cdot\,])} 
\end{eqnarray*}
be the cohomology groups arising from that complex. We call a critical point $(A,\omega)\in\Crit(\mathcal J+h_f)$ {\bf{irreducible}} if $H_{(A,\omega)}^0=0$. It is called {\bf{nondegenerate}} if $H_{(A,\omega,f)}^1=0$.

\subsection{Transversality theorem}

Throughout this section we fix a number $a>0$. 

\begin{definition}\label{def:regularorbit}\upshape
Let $m\in\N$. A pair $(\gamma,f)\in\Gamma_m\times\mathcal F_m$ is called $a$-{\bf{regular}} if the following two conditions are satisfied.
\begin{compactitem}
\item[(i)]
Every critical point $\AA=(A,\omega)\in\crit^a(\J+h_f)$ is irreducible and nondegenerate.
\item[(ii)]
Let $\AA=(A,\omega)\colon\R\to\A(P)\times\Omega^0(\Sigma,\ad(P))$ be a solution of Eq.~\eqref{pertEYM} such that $\lim_{s\to\pm\infty}(A(s),\omega(s))=(A^{\pm},\omega^{\pm})$ for a pair $(A^{\pm},\omega^{\pm})\in\Crit^a(\J+h_f)$. Then the operator $\mathcal D_{(A,\omega,0)}$ defined in \eqref{linoperator} is surjective (for every $p>1$).
\end{compactitem}
For $\gamma\in\Gamma_m$ we denote by $\mathcal F_{\reg}^a(\gamma)$ the set of maps $f\in\mathcal F_m$ such that $(\gamma,f)$ is $a$-regular.
\end{definition}

We now proceed in two steps. We first show that nondegeneracy of critical points in the sense of Definition \ref{def:regularorbit} (i) can be achieved by adding to $\J$ a perturbation $h_f$ which depends only on $A$. Thereafter we show that there exists an $a$-regular perturbation $h_{f'}$ (which now may depend on $(A,\omega)$) close to $h_f$ such that $\crit^a(\J+h_f)=\crit^a(\J+h_{f'})$. To be able to formulate our transversality theorems it is necessary to first introduce a family of seminorms on the space of perturbations. Here we follow again closely \cite{SalWeh} and define for any $k\geq1$ and perturbation $f$ the $k$-seminorm of $X_f$ as
\begin{eqnarray*}
\Vert X_f\Vert_k\coloneqq\sup_{\AA\in\tilde\A(P)}\left(\frac{\|X_f(\AA)\|_{C^k}}{1+\|\AA\|_{C^k}}+\sup_{\alpha+v\,dt\in T_{\AA}\tilde\A(P)}\frac{\|\d\!X_f(\AA)(\alpha+v\,dt)\|_{C^{k-1}}}{\|\alpha+v\,dt\|_{C^{k-1}}(1+\|\AA\|_{C^{k-1}})^{k-1}}\right).
\end{eqnarray*}
A $k$-seminorm for $Y_f$ is defined analogously. In the case where $X_f$ depends only on $A$ (and not on $\omega$), we modify the definition of $\Vert X_f\Vert_k$ accordingly.

\begin{thm}\label{thm:firsttransvers}
Let $(\gamma_0,f_0)\in\Gamma_{m_0}\times\mathcal F_{m_0}$ be such that every critical point of $\mathcal J+h_{f_0}$ in the sublevel set of $a$ is irreducible. Then, for every $k\geq1$ and $\eps>0$, there exists $n\in\N$ and a pair $(\gamma,f)\in\Gamma_n\times\mathcal F_n$ that satisfies condition (i) in Definition \ref{def:regularorbit} and such that $\|X_f-X_{f_0}\|_k<\eps$.
\end{thm}

\begin{proof} 
We split the proof into five steps.

\setcounter{step}{0}
\begin{step}\label{step:step0}
Assume the connection $A\in\A(P)$ is irreducible. Let $h_f\colon\A(P)\to\R$ be a perturbation such that the kernel of the perturbed Yang--Mills Hessian $H_{A,f}\colon\alpha\mapsto d_A^{\ast}d_A\alpha+\ast[\ast F_A\wedge\alpha]+\d\!X_f(A)\alpha$ equals $\im\big(d_A\colon\Omega^0(\Sigma,\ad(P))\to\Omega^1(\Sigma,\ad(P))\big)$. Then $h_f$ satisfies condition (i) in Definition \ref{def:regularorbit}.
\end{step}

It suffices to show that the assumptions imply that the operator $B_{f,0}$ as defined in Proposition \ref{prop:interpolatingfamily} is injective. By Proposition \ref{prop:redeigenvalueeq} applied with $\lambda=0$, injectivity of $B_{f,0}$ is equivalent to injectivity of the operator $C_{f,0}=C_f$. Since $(A,\omega)\in\crit^a(\J+h_f)$ and therefore $\omega=\ast F_A$, the operator $C_f$ takes the form of the operators $C_f^{\pm}$ considered in the proof of Lemma \ref{lem:crossingformula}. But for these injectivity has been shown under the same assumptions we made here, and hence the claim follows.

\begin{step}\label{step:step1}
Let $(\gamma_0,f_0)\in\Gamma_{m_0}\times\mathcal F_{m_0}$ satisfy the assumptions of the theorem. Then there exists $m>m_0$ and $\gamma\in\Gamma_m$ with $\gamma_i=\gamma_{0i}$ for $i=1,\ldots,m_0$ such that the following condition is satisfied. Let
\begin{eqnarray*}
\sigma(A)\coloneqq\rho(0,A)=(\rho_1(0,A),\ldots,\rho_m(0,A))\in G^m.
\end{eqnarray*} 
Then for every critical point $(A,\omega)\in\Crit^a(\J+h_{f_0})$ and every nonzero $\alpha\in\ker H_{A,f_0}$ such that $d_A^{\ast}\alpha=0$, it follows that $[d\sigma(A)\alpha]\in T_{[\sigma(A)]}(G^m/G)$ is nonzero.
\end{step}

For any fixed nonzero $\alpha\in\ker H_{A,f_0}$ such that $d_A^{\ast}\alpha=0$ it follows from statement (ii) in Proposition \ref{prop:holonomydifferential} that we can extend $\gamma_0$ to $\gamma$ such that $[d\sigma(A)\alpha]\neq0$. Because the latter property is open it remains satisfied for all $\alpha'$ sufficiently close to $\alpha$. Elliptic theory implies that the unit sphere in $\ker H_{A,f_0}|_{\ker d_A^{\ast}}$ is compact. A standard compactness argument now completes the proof of the claim.

\begin{step}\label{step:step2}
Let $m\in\N$ and $(\gamma,f)\in\Gamma_m\times\mathcal F_m$. Fix a number $p>1$ and let $q>1$ with $p^{-1}+q^{-1}=1$ denote the Sobolev exponent dual to $p$. Then surjectivity of the operator $H_{A,f}\colon W^{1,p}(\Sigma)\to L^p(\Sigma)$ is equivalent to the vanishing of every $\beta\in L^q(\Sigma)$ with the property that $\beta\in\ker H_{A,f}^{\ast}$ and 
\begin{eqnarray}\label{eq:orthim3}
\int_{\Sigma}\langle X_f(A)\wedge\ast\beta\rangle=0.
\end{eqnarray}
\end{step}

As follows from standard arguments, the range of $H_{A,f}$ is closed. Then surjectivity of $H_{A,f}$ is equivalent to injectivity of its dual operator, whence the claim.

\begin{step}\label{step:step3}
Let $(\gamma_0,f_0)$ and $m>m_0$ be as in Step \ref{step:step1}. For $k\in\N$ we denote $\mathcal F_m^k\coloneqq C^{k+1}(\D\times G^m)^G$ and consider the linear operator 
\begin{eqnarray}\label{eq:univoperator}
W^{1,p}(\Sigma)\times\mathcal F_m^k\to L^p(\Sigma)\colon(\alpha,f')\mapsto H_{A,f}\alpha+X_{f'}(A).
\end{eqnarray}
Then for every $k\in\N$ there exists $\eps>0$ such that for all $f\in\mathcal F_m$ with $\|X_{f_0}-X_f\|_k<\eps$ and for all $(A,\omega)\in\crit^a(\J+h_f)$ the operator in \eqref{eq:univoperator} is surjective.  
\end{step}
 
We first verify the claim for $f=f_0$. Assume by contradiction that there exists $(A,\omega)\in\crit^a(\J+h_{f_0})$ such that the operator in \eqref{eq:univoperator} is not surjective. Then by Step \ref{step:step2} there exists $\beta\neq0$ in $L^q(\Sigma)$ which satisfies $\beta\in\ker H_{A,f_0}^{\ast}$ as well as \eqref{eq:orthim3}. It follows from Step \ref{step:step1}   that $[d\sigma(A)\beta]\neq0$ in $T_{[\sigma(A)]}G^m/G$. This implies that the map $r\mapsto[\rho(0,A+r\beta)]$ is an embedding of a sufficiently small neighbourhood of zero. Therefore there exists a map $\hat f\in\mathcal F_m^k$ such that 
\begin{eqnarray*}
\hat f(z,\rho(z,A+r\beta))=r\chi(r)\chi(|z|)
\end{eqnarray*}
is satisfied for all $r\in\R$ and all $z\in\D$. Here $\chi\colon\R\to[0,1]$ is a suitably chosen smooth cutoff function such that $\chi$ equals $1$ near $0$ and $\supp\chi$ is contained in a sufficiently small neighbourhood of $0$. It follows that 
\begin{eqnarray*}
dh_{\hat f}(A)\beta=\frac{d}{dr}\Big|_{r=0}\int_{\D}\hat f(z,\rho(z,A+r\beta))\,d^2z=\int_{\D}\chi(|z|)\,d^2z>0.
\end{eqnarray*}
This inequality contradicts Eq.~\eqref{eq:orthim3}. The claim in the case $f=f_0$ follows. Because surjectivity is an open condition and the set $\crit^a(\J+h_{f_0})/\G(P)$ is compact it follows that the operator in \eqref{eq:univoperator} is surjective for all $f\in\mathcal F_m$ such that $\|X_f-X_{f_0}\|_k$ is sufficiently small.

\begin{step}\label{step:step4}
We prove the theorem.
\end{step}  

We first claim that for every $k\in\N$ there exists $\eps>0$ such that 
\begin{multline*}
\mathcal M^{\ast}(\mathcal F_m^{k,\eps})\coloneqq\\
\{(A,f)\in\A(P)\times\mathcal F_m\mid d_A^{\ast}F_A+X_f(A)=0,\J(A,\ast F_A)<a,\|X_f-X_{f_0}\|_k<\eps\}
\end{multline*}
is a $C^k$-Banach manifold. Namely, the linearization of the equation $d_A^{\ast}F_A+X_f(A)=0$ at each $(A,f)\in\mathcal M^{\ast}(\mathcal F_m^{k,\eps})$ is a surjective operator as was shown in Step \ref{step:step3}. It is Fredholm by standard arguments. Therefore the infinite-dimensional version of the implicit function theorem proves the claim. It follows that the projection map $\pi\colon\mathcal M^{\ast}(\mathcal F_m^{k,\eps})\to\mathcal F_m^{k,\eps}$ is a $C^k$-Fredholm map between $C^k$-Banach manifolds. The Fredholm index of its linearization $d\pi(A,f)$ is zero for all $(A,f)\in\mathcal M^{\ast}(\mathcal F_m^{k,\eps})$. Hence the Sard--Smale theorem applies and shows that the set of regular values of $\pi$ is dense in $\mathcal F_m^{k,\eps}$. For each such regular value $f$ and every $A\in\pi^{-1}(f)$ it follows from surjectivity of the operator in \eqref{eq:univoperator} and since each $d\pi(A,f)$ has Fredholm index equal to zero that the perturbed Hessian $H_{A,f}$ is surjective. This shows that $H_{(A,\omega,f)}^1=0$ for all $(A,\omega)\in\crit^a(\J+h_f)$. Because irreducibility of connections is an open condition we also have that $H_{(A,\omega)}^0=0$ for all such $(A,\omega)$ after choosing $\eps>0$ still smaller if necessary. That in fact $f$ can be chosen to be smooth follows from a density argument as in \cite[Theorem 8.3]{SalWeh}. This completes the proof of Theorem \ref{thm:firsttransvers}. 
\end{proof}

\begin{thm}\label{thm:secondtransvers}
Let $(\gamma_0,f_0)\in\Gamma_{m_0}\times\mathcal F_{m_0}$ be the perturbation as in the conclusion of Theorem \ref{thm:firsttransvers} (which was named $(\gamma,f)$ there). Then for every $k\geq1$ and $\eps>0$ there exists $m\geq1$ and an $a$-regular pair $(\gamma_1,f_1)\in\Gamma_m\times\mathcal F_m$ which satisfies
\begin{align}
\label{eq:condperturbation1}\Crit^a(\J+h_{f_0})=\Crit^a(\J+h_{f_0}+h_{f_1}),\\
\label{eq:condperturbation2}(A,\omega)\in\Crit^a(\J+h_{f_0})\quad\Longrightarrow\quad h_{f_1}(A,\omega)=0,\\
\label{eq:condperturbation3}\|X_{f_1}\|_k<\eps.
\end{align}
\end{thm}

\begin{proof} 
We proceed in three steps. At several places we identify $(A,\omega)$ with $\AA=A+\omega\,dt$ as explained above.
\setcounter{step}{0}
\begin{step}\label{step:step0thm2}
There exists $m\geq m_0$ and $\gamma\in\Gamma_m$ with $\gamma_i=\gamma_{0i}$ for $i=1,\ldots,m_0$ such that the following condition is satisfied.
Let
\begin{eqnarray*}
\sigma(\AA)\coloneqq\rho(0,A)=(\rho_1(0,\AA),\ldots,\rho_m(0,\AA))\in G^m.
\end{eqnarray*} 
For every pair of distinct critical points $(A^{\pm},\omega^{\pm})\in\crit^a(\J+h_{f_0})$ and every $(A,\omega,0)\in\widehat M_{f_0}(A^-,\omega^-,A^+,\omega^+)$ (cf.~\textsection \ref{sect:modulispFredholm} for notation) there is $s_0\in\R$ such that the following holds. First, 
\begin{eqnarray}\label{eq:critholonomy}
\sigma(\AA(s_0))\notin\{\sigma(\AA)\mid\AA=A+\omega\,dt,(A,\omega)\in\crit^a(\J+h_{f_0})\}.
\end{eqnarray}
Second, for every nonzero $\eta\in\ker\mathcal D_{(A,\omega,0)}^{\ast}$ it follows that $[d\sigma(\AA(s_0))\partial_s\AA(s_0)]$ and $[d\sigma(\AA(s_0))\eta(s_0)]$ are linearly independent vectors in $T_{[\sigma(\AA(s_0))]}G^m/G$. 
\end{step}

Fix $s_0\in\R$ such that the connection $\AA(s_0)$ is irreducible. Such a choice of $s_0$ is possible because the limit connections $(A^{\pm},\omega^{\pm})$ are irreducible by assumption. Since $(A^+,\omega^+)$ and $(A^-,\omega^-)$ are distinct it follows that $(\J+h_{f_0})(A^-,\omega^-)< (\J+h_{f_0})(\AA(s_0))<(\J+h_{f_0})(A^+,\omega^+)$. Therefore $\AA(s_0)$ is not gauge equivalent to $(A^+,\omega^+)$ or to $(A^-,\omega^-)$. It then follows from Proposition \ref{prop:holonomydifferential} (i) that there exist $m\geq m_0$ and $\gamma\in\Gamma_m$ as required and such that \eqref{eq:critholonomy} holds. Furthermore, it follows from statement (ii) in Proposition \ref{prop:holonomydifferential} that we can choose $m\in\N$ and $\gamma\in\Gamma_m$ such that in addition $d\sigma(\AA(s_0))\partial_s\AA(s_0)\notin V$. Here we let
\begin{eqnarray*}
V\coloneqq\{(\xi\sigma_i(\AA(s_0))-\sigma_i(\AA(s_0))\xi)_{i=1,\ldots,m}\mid\xi\in\mathfrak g\}
\end{eqnarray*}
denote the tangent space at $(\sigma_i(\AA(s_0))_{i=1,\ldots,m}$ to its $G$-orbit by simultaneous conjugation. It follows that $\delta\coloneqq\inf_{v\in V}\|d\sigma(\AA(s_0))\partial_s\AA(s_0)-v\|>0$. By compactness of the unit sphere in $\ker\mathcal D_{(A,\omega,0)}^{\ast}$ it suffices to prove the claim for each fixed nonzero $\eta\in\ker\mathcal D_{(A,\omega,0)}^{\ast}$. Note that $\eta(s_0)\neq0$ by a unique continuation argument. By definition of the constant $\delta$ it follows that
\begin{eqnarray}\label{eq:condnotinV}
d\sigma(\AA(s_0))(\eta(s_0)+\lambda\partial_sA(s_0))\notin V
\end{eqnarray}
for all $|\lambda|>c\coloneqq\delta^{-1}\|d\sigma(\AA(s_0))\eta(s_0)\|$. Hence for this range of $\lambda$ we see that $[d\sigma(\AA(s_0))(\eta(s_0)+\lambda\partial_sA(s_0))]\neq0$ in $T_{[\sigma(\AA(s_0))]}G^m/G$. To also show this property for $|\lambda|\leq c$ it suffices to prove it for any fixed such $\lambda$ by openess of the condition \eqref{eq:condnotinV} and compactness of the interval $[-c,c]$. We fix an arbitrary $\lambda\in[-c,c]$. Because $\eta(s_0)\neq0$ and $\partial_s\AA(s_0)\neq0$ it follows from Proposition \ref{prop:orthrelations} that $\eta(s_0)$ and $\partial_s\AA(s_0)$ are linearly independent, and hence $\eta(s_0)+\lambda\partial_sA(s_0)\neq0$.  Again by Proposition \ref{prop:orthrelations} it follows that
\begin{eqnarray*}
d_{\AA}^{\ast}(\eta(s_0)+\lambda\partial_sA(s_0))=0.
\end{eqnarray*}
Thus Proposition \ref{prop:holonomydifferential} yields the existence of a number $m_1\geq m$ and a suitable tuple $(\gamma_1,\ldots,\gamma_{m_1})\in\Gamma_{m_1}$ of embedded loops (with $\gamma_i$ as before for $i=1,\ldots,m$) such that $[d\sigma(\AA(s_0))(\eta(s_0)-\lambda\partial_s\AA(s_0))]$ does not vanish, as desired. This completes the proof of the claim.

\begin{step}\label{step:step1thm2}
For $k\in\N$ and $\eps_1>0$ we denote
\begin{eqnarray*}
\mathcal F_m^{k,\eps_1}\coloneqq\{f\in C^{k+1}(\D\times G^m)^G\mid f|_{B_{\eps_1}(C)}=0,\|f\|_{C^{k+1}}<\eps_1\},
\end{eqnarray*}
where $C\coloneqq\{(z,\rho(z,\AA))\in\D\times G^m\mid\AA=(A,\omega)\in\crit^a(\J+h_{f_0})\}$. Then for every $k\in\N$ the constant $\eps_1>0$ can be chosen such that conditions \eqref{eq:condperturbation1}, \eqref{eq:condperturbation2}, \eqref{eq:condperturbation3} are satisfied for every perturbation $f\in\mathcal F_m^{k,\eps_1}$. Furthermore, for every $(A^{\pm},\omega^{\pm})\in\crit^a(\J+h_{f_0})$ and $f\in\mathcal F_m^{k,\eps_1}$ the linear operator 
\begin{multline*}
\hat{\mathcal D}_{(A,\omega)}\colon\mathcal W^p\oplus T_f\mathcal F_m\to\mathcal L^p,\\
(\alpha,v,\psi,f')\mapsto\mathcal D_{(A,\omega,0)}(\alpha,v,\psi)+(-X_{f'}(\AA),Y_{f'}(\AA),0)
\end{multline*}
is surjective for all $(A,\omega,0)\in\widehat M_f(A^-,\omega^-,A^+,\omega^+)$.
\end{step}

It follows by definition of the space $\mathcal F_m^{k,\eps_1}$ that conditions \eqref{eq:condperturbation1}, \eqref{eq:condperturbation2}, \eqref{eq:condperturbation3} are satisfied for every sufficiently small $\eps_1<\eps$. To prove the second claim we first note that by standard arguments the range of $\hat{\mathcal D}_{(A,\omega,f)}$ is closed. Let $q>1$ with $p^{-1}+q^{-1}=1$ denote the Sobolev exponent dual to $p$. Then surjectivity of $\hat{\mathcal D}_{(A,\omega,f)}$ is equivalent to the vanishing of every $\eta\in\mathcal L^q$ orthogonal to $\im\hat{\mathcal D}_{(A,\omega,f)}$. Any such $\eta=(\eta_0,\eta_1,\eta_2)$ is contained in $\ker\hat{\mathcal D}_{(A,\omega,f)}^{\ast}$ and satisfies
\begin{eqnarray}\label{eq:orthim2}
\int_{-\infty}^{\infty}\langle X_{f'}(\AA(s)),\eta_0(s)\rangle+\langle Y_{f'}(\AA(s)),\eta_1(s)\rangle\,ds=0
\end{eqnarray}
for every $f'\in T_f\mathcal F_m$. We first prove the assertion for $f=f_0$ and a fixed pair $(A^{\pm},\omega^{\pm})\in\crit^a(\J+h_{f_0})$. Assume by contradiction that $\eta=(\eta_0,\eta_1,\eta_2)\neq0$ satisfies both  $\hat{\mathcal D}_{(A,\omega,f_0)}^{\ast}\eta=0$ and \eqref{eq:orthim2}. As follows from Step \ref{step:step0thm2} there exists $s_0\in\R$ such that $[d\sigma(\AA(s_0))\partial_s\AA(s_0)]$ and $[d\sigma(\AA(s_0))\eta(s_0)]$ are linearly independent vectors in $T_{[\sigma(\AA(s_0))]}(G^m/G)$. Therefore the map
\begin{eqnarray*}
(r,s)\mapsto[\rho(z,\AA(s)+r\eta(s))]\in G^m/G
\end{eqnarray*}
is an embedding of a neighbourhood of $(0,s_0)\in\R^2$. Now let $\chi\colon\R\to[0,1]$ be a suitably chosen smooth cutoff function such that $\supp\chi$ is contained in a sufficiently small neighbourhood of $0$ and is equal to $1$ near $0$. Then there exists a smooth $G$-invariant function $f'\colon\D\times G^m\to\R$ which by Step \ref{step:step0thm2} can be chosen to vanish in a neighbourhood of $C$ and satisfies 
\begin{eqnarray*}
f'(z,\rho(z,\AA(s)+r\eta(s)))=r\chi(r)\chi(s-s_0)\chi(|z|).
\end{eqnarray*}
for all $(z,r,s)\in\D\times\R\times\R$. For each $s\in\R$ this implies that
\begin{eqnarray*}
dh_{f'}(\AA(s))\eta(s)&=&\frac{d}{dr}\Big|_{r=0}\int_{\D}f'(z,\rho(z,\AA(s)+r\eta(s)))\,d^2z\\
&=&\chi(s-s_0)\int_{\D}\chi(|z|)\,d^2z>0.
\end{eqnarray*}
Hence there exists a perturbation $f'\in T_{f_0}\mathcal F_m$ for which \eqref{eq:orthim2} is not satisfied. This contradiction shows that $\eta=0$. Therefore the operator $\hat{\mathcal D}_{(A,\omega,f_0)}$ is surjective. Surjectivity for every pair $(A^{\pm},\omega^{\pm})\in\crit^a(\J+h_{f_0})$ and $f\in\mathcal F_m^{k,\eps_1}$ follows from a compactness argument after choosing $\eps_1$ sufficiently small.

\begin{step}\label{step:step1thm3}
We prove the theorem.
\end{step}

The statement of the theorem follows from Step \ref{step:step1thm2} by carrying over the arguments of Step \ref{step:step4} of the proof of Theorem \ref{thm:firsttransvers} to the present situation. 
\end{proof}

In the course of proof of the transversality theorems we made frequently use of the following two auxiliary results.

\begin{prop}\label{prop:holonomydifferential}
For an embedding $\gamma\colon S^1\to \Sigma\times S^1$ let $\rho_{\gamma}\colon\tilde{\mathcal A}(P)\to G$ denote the resulting holonomy map, asigning $\AA\in\tilde{\mathcal A}(P)$ its holonomy along $\gamma$. Then the following statements hold true.
\begin{compactitem}
\item[(i)]
Two connections $\AA_1=A_1+\omega_1\,dt,\AA_2=A_2+\omega_2\,dt\in\tilde{\mathcal A}(P)$ are gauge equivalent if and only if there exists $g\in G$ such that
\begin{eqnarray*}
\rho_{\gamma}(\AA_2)=g^{-1}\rho_{\gamma}(\AA_1)g
\end{eqnarray*}
for every embedding $\gamma\colon S^1\to \Sigma\times S^1$.
\item[(ii)]
Let $\AA=A+\omega\,dt\in\tilde{\mathcal A}(P)$ and $(\alpha,v)\in\Omega^1(\Sigma,\ad(P))\oplus\Omega^0(\Sigma,\ad(P))$. Then
\begin{eqnarray*}
\alpha+v\,dt\in\im\big(d_{\AA}\colon\Omega^0(\Sigma,\ad(P))\to\Omega^1(\Sigma\times S^1,\ad(P\times S^1)\big)
\end{eqnarray*}
(where we denote $d_{\AA}=d_A+[\omega\wedge\,\cdot\,]\wedge\,dt$) if and only if for every embedding $\gamma\colon S^1\to \Sigma\times S^1$ there exists $\xi\in\mathfrak g$ such that
\begin{eqnarray*}
d\rho_{\gamma}(\AA)(\alpha+v\,dt)=\rho_{\gamma}(\AA)\xi-\xi\rho_{\gamma}(\AA).
\end{eqnarray*}
\end{compactitem}
\end{prop}

\begin{proof}
Both statements follow from the fact that a connection $\AA\in\tilde\A(P)$ is uniquely determined by parallel transport along the set of embedded loops in $\Sigma\times S^1$.
\end{proof}

\begin{prop}\label{prop:orthrelations}
Assume $(A,\omega,0)$ is a solution of Eq.~\eqref{pertEYM}. Denote as before $\AA=A+\omega\,dt$ and $d_{\AA}=d_A+[\omega\wedge\,\cdot\,]\wedge\,dt$. Then $(\eta_0(s)+\eta_1(s)\,dt)\perp\im d_{\AA(s)}$ for all $\eta=(\eta_0,\eta_1,0)\in\ker\mathcal D_{(A,\omega,0)}^{\ast}$ and $s\in\R$. The same relation holds true for $\eta=\partial_s\AA$. Furthermore, $\langle\eta_0(s)+\eta_1(s)\,dt,\partial_s\AA(s)\rangle=0$ for all $\eta=(\eta_0,\eta_1,0)\in\ker\mathcal D_{(A,\omega,0)}^{\ast}$ and $s\in\R$. 
\end{prop}

\begin{proof}
Let $(\eta_0,\eta_1,0)\in\ker\mathcal D_{(A,\omega,0)}^{\ast}$. Denoting $\hat\eta\coloneqq(\eta_0,\eta_1)$ and
\begin{eqnarray*}
B_{(A,\omega)}(\alpha,v)\coloneqq(\ast[\omega\wedge\alpha]-\ast d_Av-\d\!X_f(\AA)(\alpha,v),\ast d_A\alpha-v+\d\!Y_f(\AA)(\alpha,v))
\end{eqnarray*}
this implies that $-\partial_s\hat\eta+B_{(A,\omega)}\hat\eta=0$. On the other hand, from gauge-invariance of Eq.~\eqref{pertEYM} it follows that $\partial_sd_{\AA}\ph+B_{(A,\omega)}d_{\AA}\ph=0$. Therefore, by symmetry of the operator $B_{(A,\omega)}$ it follows that
\begin{align*}
\frac{d}{ds}\langle\hat\eta,d_{\AA}\ph\rangle=&\langle\partial_s\hat\eta,d_{\AA}\ph\rangle+\langle\hat\eta,\partial_sd_{\AA}\ph\rangle\\
=&\langle B_{(A,\omega)}\hat\eta,d_{\AA}\ph\rangle+\langle\hat\eta,-B_{(A,\omega)}d_{\AA}\ph\rangle\\
=&0.
\end{align*} 
This shows that the inner product $\langle\hat\eta,d_{\AA}\ph\rangle$ is constant in $s$. Because $\hat\eta(s)$ converges to zero as $s\to\pm\infty$ this constant must be zero and the first claim follows. Let $\ph\in\Omega^0(\Sigma,\ad(P))$. Then the second claim follows from the calculation
\begin{align*}
\langle\partial_s A+\partial_s\omega\,dt,&d_A\ph+[\omega\wedge\ph]\wedge\,dt\rangle=\langle\partial_sA,d_A\ph\rangle+\langle\partial_s\omega,[\omega\wedge\ph]\rangle\\
=&\langle\ast d_A\omega+X_f(\AA),d_A\ph\rangle+\langle\omega-\ast F_A-Y_f(\AA),[\omega\wedge\ph]\rangle\\
=&\langle d_A^{\ast}\ast d_A\omega+d_A^{\ast}X_f(\AA),\ph\rangle+\langle\ast[\omega\wedge F_A]+[\omega\wedge Y_f(\AA)],\ph\rangle\\
=&0.
\end{align*}
The second equation holds by \eqref{pertEYM} and the fourth equation follows from
\begin{eqnarray*}
d_A^{\ast}X_f(\AA)=0,\qquad[\omega\wedge Y_f(\AA)]=0,
\end{eqnarray*}
together with the identity $d_A^{\ast}\ast d_A\omega=\ast d_Ad_A\omega=\ast[F_A\wedge\omega]$. This proves the second claim. Finally, because $\partial_s\AA$ satisfies the linearized equation $(\frac{d}{ds}+B_{(A,\omega)})\partial_s\AA=0$ and $\hat\eta=(\eta_0,\eta_1)$ satisfies by assumption $(-\frac{d}{ds}+B_{(A,\omega)})\hat\eta=0$ it follows by the same calculation as in the proof of the first claim that $\langle\hat\eta(s),\partial_s\AA(s)\rangle=0$ for all $s\in\R$.
\end{proof}

\section{Elliptic Yang--Mills homology}\label{sec:homology}

We fix a closed oriented Riemannian surface $(\Sigma,g)$, a compact Lie group $G$, a principal $G$-bundle $P\to\Sigma$, a regular value $a>0$ of $\J$, and an $a$-regular perturbation $h_f$. Associated to these data we define the elliptic Yang--Mills homology $HYM_{\ast}(\Sigma,P,g,a,f)$ as follows. Let
\begin{eqnarray*}
\mathcal R^a:=\frac{\crit^a(\J+h_f)}{\G(P)}
\end{eqnarray*}
denote the (finite) set of gauge equivalence classes of critical points of $\J+h_f$ of energy less than $a$. These generate a chain complex
\begin{eqnarray*}
CYM_{\ast}^a(\Sigma,P,g,f):=\sum_{[(A,\omega)]\in\mathcal R^a}\mathbbm Z_2\langle[(A,\omega)]\rangle
\end{eqnarray*}
with grading given by the index $\ind H_{A,f}$ of the perturbed Yang--Mills Hessian (cf.~Theorem \ref{thm:Fredholmthm} for further details). For a pair $[(A^-,\omega^-],[(A^+,\omega^+)]\in\mathcal R^a$ of critical points we let 
\begin{eqnarray*}
\mathcal M_f(A^-,\omega^-,A^+,\omega^+)=\frac{\widehat{\mathcal M}_f(A^-,\omega^-,A^+,\omega^+)}{\G(P)}
\end{eqnarray*}
be the moduli space as defined in \textsection \ref{sect:modulispFredholm}. It is a smooth manifold of dimension $\ind H_{A^-,f}-\ind H_{A^+,f}$. The group $\R$ acts on it freely by time-shifts. If this index difference equals $1$ it follows by compactness (cf.~Theorem \ref{thm:compactness}) that the quotient of $\mathcal M_f(A^-,\omega^-,A^+,\omega^+)$ modulo the $\R$-action consists of a finite number of points. For each $k\in\mathbbm N_0$ we hence obtain a well-defined boundary operator
\begin{eqnarray*}
\partial_k\colon CYM_k^a(\Sigma,P,g,f)\to CYM_{k-1}^a(\Sigma,P,g,f)
\end{eqnarray*}
as the linear extension of the map
\begin{eqnarray*}
\partial_kx\coloneqq\sum_{x'\in\mathcal R^a\atop\ind(x')=k-1}n(x,x')x',
\end{eqnarray*}
$x\in\mathcal R^a$ a critical point of index $\ind x=k$. Here $n(x,x')\in\mathbbm Z_2$ is defined for $x=[(A^-,\omega^-)]$ and $x'=[(A^+,\omega^+)]$ as the number of $\R$-equivalence classes of elements in $\mathcal M_f(A^-,\omega^-,A^+,\omega^+)$, counted modulo $2$, i.e.~   
\begin{eqnarray*}
n(x,x')\coloneqq\#\frac{\mathcal M_f(A^-,\omega^-,A^+,\omega^+)}{\R}\quad (\modulo 2).
\end{eqnarray*}

\begin{lem}\label{lem:cascadechainmap}
The sequence $\partial_{\ast}$ of homomorphisms satisfies $\partial_{\ast}\circ\partial_{\ast+1}=0$, i.e.~$(CYM_{\ast}^a(\Sigma,P,g,f),\partial_{\ast})$ is a chain complex.
\end{lem}

\begin{proof}
The property of a chain map follows from standard arguments invoking Theorems  \ref{thm:compactness} (compactness) and \ref{thm:mainexpconvergence} (exponential decay of finite energy gradient flow lines).  
\end{proof}

We define the {\bf{elliptic Yang--Mills homology at level $a$ associated with the data}} $(\Sigma,P,g,f)$ to be the collection of abelian groups
\begin{eqnarray*}
HYM_k^a(\Sigma,P,g,f)\coloneqq\frac{\ker\partial_k}{\im\partial_{k+1}}\qquad(k\in\mathbbm N_0).
\end{eqnarray*}
We conclude this section with a number of remarks.
\begin{compactitem}
\item
We expect the usual cobordism arguments to show independence of the homology groups $HYM_{\ast}^a(\Sigma,P,g,f)$ on the choice of the Riemannian metric $g$ and $a$-regular perturbation $f$ used to define it.
\item
To keep the exposition short we developed here a version of elliptic Yang--Mills homology using coefficients in $\Z_2$. Defining a variant of it with coefficient ring $\Z$ requires to deal with oriented moduli spaces.
\item
It would be very interesting to compare elliptic Yang--Mills homology with the Morse homology obtained from the Yang--Mills gradient flow on a Riemannian surface and first considered by Atiyah and Bott \cite{AB}. The chain complexes in both cases are generated by (perturbed) Yang--Mills connections $A$ (respectively pairs $(A,\omega)$ where $A$ is a (perturbed) Yang--Mills connection and $\omega=\ast F_A$), and hence coincide. However the boundary operators are very different, involving a parabolic equation in the classical situation, in contrast to the elliptic system in the approach presented here. We shall comment more on this point in \textsection \ref{sect:RelshipwithparabolicYM}.
\end{compactitem}

\section{Relationship to parabolic Yang--Mills Morse homology}\label{sect:RelshipwithparabolicYM}

We now discuss a partly conjectural relationship between elliptic Yang--Mills homology and the by now classical Morse homology theory based on the $L^2$-gradient flow of the Yang--Mills functional over Riemann surfaces. The latter will be called {\bf{parabolic Yang--Mills homology}} for brevity. We refer to the seminal article \cite{AB} by Atiyah and Bott which initiated (amongst other things) the study of the Yang--Mills functional on Riemann surfaces from a Morse theoretical perspective. Recall that the (unperturbed) Yang--Mills functional is the map
\begin{eqnarray*}
\YM\colon\A(P)\to\R,\qquad\YM(A)=\frac{1}{2}\int_{\Sigma}|F_A|^2\,\dvol_{\Sigma}.
\end{eqnarray*}
The following observation provides the starting point of our discussion.
 
\begin{lem}\label{lem:increaseenergy}
For all $(A,\omega)\in\A(P)\times\Omega^0(\Sigma,\ad(P))$ there holds the energy inequality
\begin{eqnarray}\label{eq:increaseenergy}
\mathcal J(A,\omega)\leq\YM(A),
\end{eqnarray}
with equality if and only if $\omega=\ast F_A$.
\end{lem}

\begin{proof}
The claim follows, using the Cauchy--Schwartz inequality, from
\begin{multline*}
\J(A,\omega)=\int_{\Sigma}\langle F_A,\ast\omega\rangle-\frac{1}{2}|\omega|^2\,\dvol_{\Sigma}\leq\int_{\Sigma}|F_A||\omega|-\frac{1}{2}|\omega|^2\,\dvol_{\Sigma}\\
\leq\int_{\Sigma}\frac{1}{2}|F_A|^2\,\dvol_{\Sigma}=\YM(A).
\end{multline*} 
The case of equality can easily be read off from this estimate.
\end{proof}

We now consider the $\G(P)$-equivariant map
\begin{eqnarray*}
\theta\colon\A(P)\to\A(P)\times\Omega^0(\Sigma,\ad(P)),\qquad A\mapsto(A,\ast F_A).
\end{eqnarray*}
As initially remarked, a pair $(A,\omega)$ is a critical point of $\J$ if and only if $A$ is a Yang--Mills connection and $\omega=\ast F_A$. Hence the map $\theta$ induces a $\G(P)$-equivariant bijection between the sets $\crit(\YM)$ and $\crit(\J)$. By Lemma \ref{lem:increaseenergy} it decreases energy in the sense that
\begin{eqnarray}\label{eq:increaseenergy}
\J\circ\theta(A)\leq\YM(A)
\end{eqnarray}
for all $A\in\A(P)$. This observation plays a crucial role in the following construction. The idea is to couple solutions of the Yang--Mills gradient flow equation
\begin{eqnarray}\label{eq:L2YMG}
0=\partial_sA+d_A^{\ast}F_A
\end{eqnarray}
on the half-infinite interval $(-\infty,0]$ with solutions of the (unperturbed) equation \eqref{pertEYM1} on the half-infinite interval $[0,\infty)$. We are thus lead to define for critical points $A^-\in\crit(\YM)$ and $(B^+,\omega^+)\in\crit(\J)$ the set
\begin{multline*}
\widehat{\mathcal M}_{\textrm{hyb}}(A^-;B^+,\omega^+)\coloneqq\\
\Big\{(A;B,\omega)\mid A\,\textrm{satisfies}\,\eqref{eq:L2YMG}\, \textrm{on} \,(-\infty,0)\times \Sigma,\lim_{s\to-\infty}A(s)=A^-,\\
(B,\omega)\,\textrm{satisfies}\,\eqref{pertEYM1}\, \textrm{on} \,(0,\infty)\times \Sigma,\lim_{s\to\infty}(B(s),\omega(s))=(B^+,\omega^+),\\
(B(0),\omega(0))=\theta(A(0))\Big\},
\end{multline*}
and call the quotient
\begin{eqnarray*}
\mathcal M_{\hybr}(A^-;B^+,\omega^+)\coloneqq\frac{\widehat{\mathcal M}_{\textrm{hyb}}(A^-;B^+,\omega^+)}{\G(P)}
\end{eqnarray*}
{\bf{hybrid moduli space}} (of connecting trajectories between $A^-$ and $(B^+,\omega^+)$). We refer to the articles \cite{AbSchwarz,Swoboda2} for related constructions of hybrid moduli spaces in various different contexts. The significance of the moduli spaces $\mathcal M_{\hybr}(A^-;B^+,\omega^+)$ introduced here is that they are supposed to give rise to a chain map $\Theta$ between the Morse complexes defined from the elliptic and the parabolic Yang--Mills flow equations. Namely, after introducing a suitable perturbation scheme (which for the Yang--Mills gradient flow has been done in \cite{Swoboda1}) we expect these hybrid moduli spaces to be smooth manifolds of dimension equal to the index difference $\ind H_{A^-}-\ind H_{B^+}$ of the corresponding Yang--Mills Hessians. As in \cite{AbSchwarz,Swoboda2} they are expected to admit compactifications by adding configurations of broken trajectories. In particular, in the case of index difference equal to zero, $\mathcal M_{\hybr}(A^-;B^+,\omega^+)$ is a finite set and we may then define $\#\,\mathcal M^{\hybr}(A^-;B^+,\omega^+)$ to be the number of elements of $\mathcal M^{\hybr}(A^-;B^+,\omega^+)$ (again counted modulo $2$ to avoid orientation issues). For the remainder of this section we fix a regular value $a>0$ of $\YM$ and let $\crit^a(\YM)$ denote the set of critical points $A$ of $\YM$ with $\YM(A)\leq a$. By the above considerations, we obtain a well-defined map $\Theta\coloneqq(\Theta_k)_{k\in\mathbbm N_0}$, where the homomorphism $\Theta_k$ is defined by linear extension of the map   
\begin{eqnarray*}
\Theta_k(A^-)\coloneqq\sum_{(B^+,\omega^+)\in\crit^a(\J)\atop\ind H_{B^+}=k}\#\,\mathcal M_{\hybr}(A^-;B^+,\omega^+)\cdot(B^+,\omega^+),
\end{eqnarray*} 
where $A^-\in\crit^a(\YM)$ with $\ind H_{A^-}=k$. By standard arguments, the map $\Theta$ is a chain map between the Morse complexes generated by the sets $\crit^a(\YM)$ and $\crit^a(\J)$. The crucial observation, first made in a related situation in \cite{AbSchwarz}, cf.~also \cite{Swoboda2}, is that the homomorphism $\Theta$ in each degree $k$ can be represented by an invertible upper-triangular matrix of the form
\begin{eqnarray*}
(\Theta_{ij}^k)_{1\leq i,j\leq m}=\left(\begin{array}{ccccc} 1&\ast&\cdots&\cdots&\ast\\0& 1&\ddots&&\vdots\\\vdots&\ddots&\ddots&\ddots&\vdots\\\vdots&&\ddots& 1&\ast\\0&\cdots&\cdots&0& 1\end{array}\right)\in\Z_2^{m\times m},
\end{eqnarray*}
and hence is invertible in $\Z_2^{m\times m}$. The existence of such an invertible upper-triangular matrix $(\Theta_{ij}^k)_{1\leq i,j\leq m}$ can be seen as follows. As a set of generators of the chain group of degree $k\in\N_0$ of the parabolic Yang--Mills Morse complex we choose the critical points $(A_1,\ldots,A_m)\in\crit^a(\YM)$  of index $\ind H_{A_j}=k$, ordered by increasing action. We then define the $m$-tuple
\begin{eqnarray*}
((B_1,\omega_1),\ldots,(A_m,\omega_m))\coloneqq(\theta(A_1),\ldots,\theta(A_m)),
\end{eqnarray*}
which as explained above is a set of generators of the chain group of degree $k$ of the elliptic Yang--Mills Morse complex. Inequality \eqref{eq:increaseenergy} (which in this case is an equality) moreover shows that this $m$-tuple is ordered by increasing action, too. It follows from Lemma \ref{lem:increaseenergy} and inequality \eqref{eq:L2YMG} by an elementary argument that
\begin{eqnarray}\label{eq:emptymodspace}
\mathcal M_{\hybr}(A_i^-;B_j^+,\omega_j^+)=\emptyset\qquad\textrm{if}\quad i<j.
\end{eqnarray}
Namely, along any flow line in $\widehat{\mathcal M}_{\hybr}(A_i^-;B_j^+,\omega_j^+)$ energy is strictly increasing and therefore inequality \eqref{eq:L2YMG} implies \eqref{eq:emptymodspace}. Furthermore, each moduli space $\mathcal M_{\hybr}(A_i^-;B_i^+,\omega_i^+)$ is represented by precisely one element, the concatenation of the constant flow lines $A\equiv A_i^-$ and $(B,\omega)\equiv(B_i^+,\omega_i^+)$. The invertible chain homomorphism $\Theta^k$ descends in each degree $k\in\N_0$ to an isomorphism between the $k$-th elliptic and parabolic Yang--Mills homology groups (at level $a$). The details of the construction of this isomorphism are similar to those being carried out in \cite{AbSchwarz,Swoboda2}. We leave them to a forthcoming publication.

\section{Three-dimensional product manifolds}\label{sec:special case}

So far, we defined the elliptic Yang--Mills homology in the case $n=2$ which allows the equations to be nicely elliptic. In the subsequent sections we discuss two extensions of this theory to the case $n\geq3$. First we explain how two-dimensional elliptic Yang--Mills homology may be applied to the special case of three-dimensional products $Y=\Sigma\times S^1$. Then in \textsection \ref{sec:restrflow} we show how to restore ellipticity by restricting the space $X=\A(P)\times\Omega^{n-2}(Y,\ad(P))$ to a suitable Banach submanifold.\\
\noindent\\
We start by recalling some facts concerning the moduli space $\mathcal M^{\gamma}(P)$ of flat connections on a principal bundle $P$ over a Riemann surface $(\Sigma,g_{\Sigma})$ of genus $\gamma$. It was investigated for the first time in 1983 by Atiyah and Bott in \cite{AB}. They noticed that the conformal structure of $\Sigma$ gives rise to a Riemannian metric and an almost complex structure on $\mathcal M^{\gamma}(P)$; moreover, if a nontrivial principal $\SO(3)$-bundle $P$ is chosen, then the moduli space $\mathcal M^{\gamma}(P)$, defined as the quotient of the space of the flat connections $\mathcal A_0(P)\subseteq\mathcal A(P)$ and the identity component $\mathcal G_0(P)$ of the group of gauge transformations, is a compact K\"ahler manifold without singularities of dimension $6\gamma-6$ (cf.~\cite{DostSal1}). In the nineties some aspects of the topology of $\mathcal M^{\gamma}(P)$ were investigated by Dostoglou and Salamon (cf.~\cite{DostSal1,DostSal2,DostSal3,Sal3}), who proved among other results the Atiyah--Floer conjecture in this context, as well as by Hong (cf.~\cite{Hong}).\\
\noindent\\ 
We fix a three-dimensional product manifold $Y=\Sigma \times S^1$, $(\Sigma,g_{\Sigma})$ a Riemann surface of genus $\gamma$, with the partially rescaled product metric $\varepsilon^2 g_\Sigma \oplus g_{S^1}$, for a parameter $\varepsilon>0$. We include in the discussion here only nontrivial principal $\SO(3)$-bundles P and $\hat P=P\times S^1$ over $\Sigma$ and $Y$, respectively. The motivation for considering this setup comes from the fact that, as we will see, the perturbed Yang--Mills connections on $P$ and the resulting elliptic Yang--Mills homology groups are strongly related to the perturbed geodesics and the homology of the free loop space $\mathcal L\mathcal M^{\gamma}(P)$ associated with the moduli space of flat connections $\mathcal M^{\gamma}(P)$. In fact, by results of the first author (cf.~\cite{Janner3}), there is a bijection between the perturbed Yang--Mills connections on the bundle $\hat P$ and the perturbed geodesics on $\mathcal M^{\gamma}(P)$. Furthermore, he showed that the Morse homologies, defined from the perturbed parabolic Yang--Mills gradient flow on $\hat P \to \Sigma \times S^1$ and the heat flow on $\mathcal M^{\gamma}(P)$, are isomorphic, provided that $\varepsilon$ is small enough and that an energy bound $b$ is chosen (cf.~\cite{Janner,Janner2}). Hence we have:
\begin{equation}\label{eq:isomorphism}
HM_*\big(\mathcal L^b\mathcal M^{\gamma}(P),\mathbb Z_2\big)\cong HM_*\big(\mathcal A^{\varepsilon,b}(\hat P)/\mathcal G_0(\hat P),\mathbb Z_2\big)
\end{equation}
where $\mathcal L^b\mathcal M^{\gamma}(P)\subseteq \mathcal L\mathcal M^{\gamma}(P)$ and $\mathcal A^{\varepsilon,b}(\hat P)\subseteq \mathcal A(\hat P)$, respectively, denote the subsets with energies bounded from above by $b$.\\
\noindent\\
The heat flow homology appearing on the left-hand side of \eqref{eq:isomorphism}  is well-defined in any dimension due to work of Salamon and Weber (cf.~\cite{SalWeb,Web1,Web2}). In contrast, the Yang--Mills gradient flow exists only if the base manifold $Y$ is two- or three-dimensional, or if it has a symmetry of codimension three (cf.~\cite{Rade,Rade2}). Furthermore, apart from the two-dimensional case, the Morse--Smale transversality is not yet proven and thus the homology $HM_{\ast}\big(\mathcal A^{\varepsilon,b}(P)/\mathcal G_0(P),\Z_2\big)$ is presently not defined in the general case. One notable exception are products $Y=\Sigma\times S^1$. Here the gauge equivalence classes of gradient flow lines are in a bijective relation with the gauge equivalence classes of heat gradient flow lines on $\mathcal L\mathcal M^{\gamma}(P)$ provided that $\varepsilon$ is small enough and thus, the homology $HM_{\ast}\big(\mathcal A^{\varepsilon,b}(P)/\mathcal G_0(P),\Z_2\big)$ is well defined.\\
\noindent\\
As follows from works of Viterbo (cf.~\cite{Viterbo}), Salamon and Weber (cf.~\cite{SalWeb}) and Abbondandolo and Schwarz (cf.~\cite{AbSchwarz,AbSchwarz1}), the Morse homology of $\mathcal L^b\mathcal M^{\gamma}(P)$ is isomorphic to Floer homology of the cotangent bundle $T^*\mathcal M^{\gamma}(P)$ defined via the Hamiltonian $H_V$ given by the sum of kinetic and potential energy and considering only orbits with action bounded by $b$.  Moreover, Weber (cf.~\cite{Web1}) showed that the Morse homology of the free loop space defined from the heat flow is isomorphic to its singular homology. Summarizing we therefore have (in part conjecturally) the following isomorphisms of abelian groups:
\begin{equation*}
\begin{array}{ccc}
H_*\big(\mathcal L^b\mathcal M^{\gamma}(P)\big) &&H_*\big(\mathcal A^{\varepsilon,b}(\hat P)/\mathcal G_0(\hat P)\big)\\
\cong&&\stackrel {?}\cong\\
HM_*\big(\mathcal L^b\mathcal M^{\gamma}(P)\big)&\cong& HM_*\big(\mathcal A^{\varepsilon,b}(\hat P)/\mathcal G_0(\hat P)\big)\\
\cong&&\\
HF^b_*\big(T^*\mathcal M^{\gamma}(P), H_V\big)&&
\end{array}
\end{equation*}
It is still an open question (marked with an interrogation point in the above diagram) whether the Morse homology of $\mathcal A^{\varepsilon,b}(\hat P)/\mathcal G_0(\hat P)$ defined from the $L^2$-gradient flow of the Yang--Mills functional is isomorphic to the singular homology. We may also ask how our newly defined elliptic Yang--Mills homology fits into this picture. To be specific, we consider the elliptic Yang--Mills homology $HYM^b(\Sigma\times S^1,\varepsilon^2 g_\Sigma \oplus g_{S^1},f)$ defined from the functional $\mathcal J+h_f$ on $\mathcal A(\hat P)\times \Omega^1(\Sigma\times S^1,\ad(P))$, which can be seen as the Yang--Mills analogue of the cotangent bundle. In this special case it might be possible to show a bijective relation between the gauge equivalence classes of the elliptic Yang-Mills flow lines between stationary points and those of the parabolic Yang-Mills or of the Floer homology setup, provided that $\varepsilon$ is small enough, by an analogous argument as in \cite{Janner} or in \cite{SalWeb}. We leave the proof to a forthcoming publication. This would then define $HYM^b(\Sigma\times S^1,\varepsilon^2 g_\Sigma \oplus g_{S^1},f)$ and lead to an isomorphismus between it, $HM_*\big(\mathcal A^{\varepsilon,b}(\hat P)/\mathcal G_0(\hat P)\big)$ and $HF^b_*\big(T^*\mathcal M^{\gamma}(P), H_V\big)$, considering $\Z_2$ coefficients homologies:
\begin{equation*}
\begin{array}{ccc}
H_*\big(\mathcal L^b\mathcal M^{\gamma}(P)\big) &&\\
\cong&&\\
HM_*\big(\mathcal L^b\mathcal M^{\gamma}(P)\big)&\cong& HM_*\big(\mathcal A^{\varepsilon,b}(\hat P)/\mathcal G_0(\hat P)\big)\\
\cong&& \stackrel {?}\cong\\
HF^b_*\big (T^*\mathcal M^{\gamma}(P), H_V\big)&\stackrel {?}\cong& HYM^b(\Sigma\times S^1,\varepsilon^2 g_\Sigma \oplus g_{S^1},f)
\end{array}
\end{equation*}

\section{Restricted flow}\label{sec:restrflow}

In this final section we discuss a modification and extension of Eq.~\eqref{pertEYM1} to base manifolds $Y$ of dimension $n\geq3$. It is motivated from the shortcoming that the linearization of Eq.~\eqref{pertEYM1} is an elliptic system only in dimension $n=2$, even after imposing gauge-fixing conditions.

\begin{prop}\label{prop:ellipticsystem}
Let $I$ be an open interval and $(A,\omega,\Psi)\in C^{\infty}(I,\A(P)\times\Omega^{n-2}(Y,\ad(P))\times\Omega^0(Y,\ad(P)))$. Then the linear operator $\nabla_s+B_{(A,\omega,\Psi)}$ associated with $(A,\omega,\Psi)$ (cf.~\eqref{linoperator}) is elliptic if and only if $n=2$. 
\end{prop}

\begin{proof}
It is straightforward to verify the assertion from a calculation of the principal symbol of the operator $\nabla_s+B_{(A,\omega,\Psi)}$. As this requires to introduce some notation we instead prove the claim for the operator  
\begin{eqnarray*}
(\nabla_s+B_{(A,\omega,\Psi)})(\nabla_s+B_{(A,\omega,\Psi)})^{\ast}=-\nabla_s^2+B_{(A,\omega,\Psi)}^2.
\end{eqnarray*}
A short calculation shows that
\begin{eqnarray*}
-\nabla_s^2+B_{(A,\omega,\Psi)}^2=\begin{pmatrix}-\nabla_s^2+\Delta_A&0&0\\0&-\nabla_s^2+d_A^{\ast}d_A&0\\0&0&-\nabla_s^2+\Delta_A\end{pmatrix}+R_{(A,\omega,\Psi)},
\end{eqnarray*}
where $R_{(A,\omega,\Psi)}$ is a differential operator of order one. In the case $n=2$ it follows that the entry $d_A^{\ast}d_A$ appearing in the above matrix is equal to $\Delta_A$. Then all the three components of the leading order term of $-\nabla_s^2+B_{(A,\omega,\Psi)}^2$ are Laplacians on one of the bundles $\Omega^1(I\times Y,\ad(I\times P))$ or $\Omega^0(I\times Y,\ad(I\times P))$, showing that this operator is elliptic. This fails to be the case if $n\geq3$ and hence the claim follows.
\end{proof}

To overcome the lack of ellipticity of its linearization in dimension $n\geq3$ we introduce the following modification of Eq.~\eqref{pertEYM1}. This modification arises as the formal $L^2$-gradient flow equation associated with the restriction of the functional $\J$ to the infinite-dimensional manifold (with singularities)
\begin{eqnarray*}
X_1\coloneqq\{(A,\omega)\in X\mid d_A^{\ast}\omega=0\},
\end{eqnarray*}
where we let $X\coloneqq\A(P)\times\Omega^{n-2}(Y,\ad(P))$. Note that in dimension $n=2$ it follows that $X_1=X$, and the system \eqref{pertEYM1} studied so far appears as a special case of the restricted flow equations which we shall introduce next. The tangent space of $X_1$ at $(A,\omega)\in X_1$ is $T_{(A,\omega)}X_1=\ker L_{(A,\omega)}$, where we define
\begin{multline*}
L_{(A,\omega)}\colon\Omega^1(Y,\ad(P))\oplus\Omega^{n-2}(Y,\ad(P))\to\Omega^{n-3}(Y,\ad(P)),\\
(\alpha,v)\mapsto d_A^{\ast}v+\ast[\ast\omega\wedge\alpha].
\end{multline*}
The formal adjoint of $L_{(A,\omega)}$ is the operator
\begin{multline*}
L_{(A,\omega)}^{\ast}\colon\Omega^{n-3}(Y,\ad(P))\to\Omega^1(Y,\ad(P))\oplus\Omega^{n-2}(Y,\ad(P)),\\
\Psi\mapsto((-1)^{n+1}\ast[\ast\omega\wedge\Psi],d_A\Psi).
\end{multline*}
Let $(A,\omega)\in X_1$. Because the condition $d_A^{\ast}\omega=0$ is gauge-invariant it follows that the image of the map 
\begin{eqnarray*}
d_{(A,\omega)}\colon\Omega^0(Y,\ad(P))\to\Omega^1(Y,\ad(P))\oplus\Omega^{n-2}(Y,\ad(P)),\qquad\ph\mapsto(d_A\ph,[\omega\wedge\ph]) 
\end{eqnarray*}
is contained in $\ker L_{(A,\omega)}$. We therefore obtain an elliptic complex
\begin{eqnarray*}
0\longrightarrow\Omega^0(Y,\ad(P))\overset{d_{(A,\omega)}}{\longrightarrow}\Omega^1(Y,\ad(P))\oplus\Omega^{n-2}(Y,\ad(P))\overset{L_{(A,\omega)}}{\longrightarrow}\Omega^{n-3}(Y,\ad(P))\longrightarrow0
\end{eqnarray*} 
together with a Laplace operator $D_{(A,\omega)}\coloneqq d_{(A,\omega)}d_{(A,\omega)}^{\ast}+L_{(A,\omega)}^{\ast}L_{(A,\omega)}$. Any $\xi=(\alpha,v)\in\Omega^1(Y,\ad(P))\oplus\Omega^{n-2}(Y,\ad(P))$ admits a unique $L^2$-orthogonal decomposition
\begin{eqnarray*}
\xi=\xi_0+\xi_1+d_{(A,\omega)}\ph,
\end{eqnarray*}
where $\xi_0\in\ker D_{(A,\omega)}$, $\xi_1\in\im L_{(A,\omega)}^{\ast}$, and $\ph\in\Omega^0(Y,\ad(P))$. The $L^2$-orthogonal projection $\pi_{(A,\omega)}:T_{(A,\omega)}X\to T_{(A,\omega)}X_1$ is therefore given by
\begin{eqnarray*}
\pi_{(A,\omega)}\xi=\xi-\xi_1,
\end{eqnarray*}
where $\xi_1$ is the unique solution of the elliptic equation
\begin{eqnarray*}
D_{(A,\omega)}\xi_1=L_{(A,\omega)}^{\ast}L_{(A,\omega)}\xi.
\end{eqnarray*}
If in particular $\xi=((-1)^n\ast d_A\omega,\omega-\ast F_A)$, which is the negative of the $L^2$-gradient of $\J$, then it follows by a short calculation that
\begin{eqnarray*}
L_{(A,\omega)}^{\ast}L_{(A,\omega)}\xi=(-\ast[\ast\omega\wedge\ast[\ast\omega\wedge\ast d_A\omega]],(-1)^nd_A\ast[\ast\omega\wedge\ast d_A\omega]).
\end{eqnarray*}
In this case, $\xi_1=(\alpha_1,v_1)$ is the solution of the equation
\begin{eqnarray}\label{eq:correctionterm}
D_{(A,\omega)}\xi_1=(-\ast[\ast\omega\wedge\ast[\ast\omega\wedge\ast d_A\omega]],(-1)^nd_A\ast[\ast\omega\wedge\ast d_A\omega]).
\end{eqnarray}

\begin{definition}\upshape
Let $I$ be an interval. The {\bf{restricted $L^2$-gradient flow}} is the system of equations
\begin{eqnarray}\label{restrEYM}
\begin{cases}
0=\partial_sA-d_A\Psi+(-1)^{n+1}\ast d_A\omega-\alpha_1\\
0=\partial_s\omega+[\Psi,\omega]-\omega+\ast  F_A-v_1
\end{cases}
\end{eqnarray}
for a triple $(A,\omega,\Psi)\in C^{\infty}(I,X_1\times\Omega^0(Y,\ad(P)))$. Here the correction term $\xi_1=(\alpha_1,v_1)$ is defined to be the solution of Eq.~\eqref{eq:correctionterm}. The term $\Psi$ is included in order to make \eqref{restrEYM} invariant under time-dependent gauge transformations.
\end{definition}

We conclude this section with a number of remarks. First one should note that the additional condition $d_A^{\ast}\omega=0$ imposed on pairs $(A,\omega)\in X$ is compatible with the critical point equation \eqref{statEYM}. Namely then, $\omega=\ast F_A$ and hence $d_A^{\ast}\omega=0$ holds automatically by the Bianchi identity. Second, the term $((-1)^{n+1}\ast d_A\omega-\alpha_1,-\omega+\ast  F_A-v_1)$ appearing in \eqref{restrEYM} can be understood as the Riemannian gradient of the functional $\J\colon X_1\to\R$, where we view $X_1$ as submanifold of $X$ endowed with the $L^2$ submanifold metric. Finally, the linearization of \eqref{restrEYM} together with a gauge-fixing condition as discussed in Remark \ref{rem:gaugefix} now leads to an elliptic system in any dimension $n$ (in contrast to the unrestricted flow \eqref{pertEYM1}, cf.~Proposition \ref{prop:ellipticsystem}). A slight complication now comes from the fact that the term $\xi_1=(\alpha_1,v_1)$ in \eqref{restrEYM} is nonlocal (however of order zero). Still it is conceivable that an analysis of the moduli spaces of solutions of \eqref{restrEYM} is possible in analogy to that of the solutions of \eqref{pertEYM1}. This will lead to an elliptic Yang--Mills homology for manifolds of dimension $n\geq3$. We leave this programe to be carried out in a future publication.

\appendix

\section{Differential inequalities for the energy density}

Recall from \textsection \ref{sect:compactness} the gauge-invariant energy density of the solution $(A,\omega,\Psi)$ of \eqref{pertEYM1}, which we defined to to be the function
\begin{multline*}
e(A,\omega,\Psi)\coloneqq\frac{1}{2}(|\ast d_A\omega+d_A\Psi+X_f(A,\omega)|^2\\
+|\ast F_A-\omega+[\Psi,\omega]+Y_f(A,\omega)|^2)\colon I\times\Sigma\to\R.
\end{multline*}

Our aim is to derive an elliptic differential inequality satisfied by $e(A,\omega,\Psi)$. Let $\Delta_{\Sigma}$ denote the (positive semidefinite) Laplace--Beltrami operator on $(\Sigma,g)$, and define 
\begin{eqnarray*}
\Delta_{I\times \Sigma}:=-\frac{d^2}{ds^2}+\Delta_{\Sigma}.
\end{eqnarray*}

\begin{prop}\label{prop:energyineq}
Assume that $(A,\omega,\Psi)$ is a solution of \eqref{pertEYM}. Then the energy density $e=e(A,\omega,\Psi)$  satisfies on $I\times \Sigma$ the differential inequality
\begin{eqnarray}\label{eq:estenergyineq}
\Delta_{I\times \Sigma}e\leq A_0+A_1|\omega|^6+A_2e+A_3e^{\frac{3}{2}}
\end{eqnarray}
for positive constants $A_0,A_1,A_2,A_3$ which do not depend on $(A,\omega,\Psi)$.
\end{prop}

\begin{rem}\upshape
Note that the exponent $\frac{3}{2}$ appearing in the differential inequality \eqref{eq:estenergyineq} is critical in dimension $\dim Y=3$ but subcritical in the present case.
\end{rem}

\begin{proof}
Since the asserted inequality is invariant under gauge transformations we may assume that $\Psi=0$. We first show the claim in the unperturbed case $(X_f,Y_f)=0$. By differentiating \eqref{pertEYM} with respect to $s$ we obtain
\begin{eqnarray}\label{EYM2}
\left\{\begin{array}{rcl}\ddot A&=&d_A^{\ast}F_A+\ast d_A\omega+\ast[\dot A\wedge\omega]\\
\ddot\omega&=&d_A^{\ast}d_A\omega+\dot\omega\end{array}\right.
\end{eqnarray}
It then follows that
\begin{align*}
\lefteqn{\Delta_{I\times \Sigma}e}\\
&=-|\frac{d}{ds}d_A\omega|^2-|\frac{d}{ds}(-\omega+\ast F_A)|^2-|\nabla_Ad_A\omega|^2-|\nabla_A(-\omega+\ast F_A)|^2\\
&+ \langle d_A\omega,(-\frac{d^2}{ds^2}+\nabla_A^{\ast}\nabla_A)d_A\omega\rangle + \langle-\omega+\ast F_A,(-\frac{d^2}{ds^2}+\nabla_A^{\ast}\nabla_A)(-\omega+\ast F_A)\rangle \\
&=-|\frac{d}{ds}d_A\omega|^2-|\frac{d}{ds}(-\omega+\ast F_A)|^2-|\nabla_Ad_A\omega|^2-|\nabla_A(-\omega+\ast F_A)|^2\\
&+\underbrace{\langle d_A\omega,(-\frac{d^2}{ds^2}+\Delta_A)d_A\omega+\{F_A,d_A\omega\}+\{R_{\Sigma},d_A\omega\} \rangle}_{\mathord{\mathrm{I}}}\\
&+\underbrace{\langle-\omega+\ast F_A,(-\frac{d^2}{ds^2}+\Delta_A)(-\omega+\ast F_A)+\{F_A,-\omega+\ast F_A\}+\{R_{\Sigma},-\omega+\ast F_A\}\rangle}_{\mathord{\mathrm{II}}}.
\end{align*}
In the last step we replaced $\nabla_A^{\ast}\nabla_A$ using the standard Weizenb\"ock formula
\begin{eqnarray*}
\nabla_A^{\ast}\nabla_A=\Delta_A+\{F_A,\,\cdot\,\}+\{R_{\Sigma},\,\cdot\,\},
\end{eqnarray*}
where the brackets denote bilinear terms with fixed coefficients, and $R_{\Sigma}$ is some curvature expression determined by the fixed Riemannian metric $g$. It remains to analyze terms $\mathord{\mathrm{I}}$ and $\mathord{\mathrm{II}}$. For the first one we obtain, making use of Eq.~\eqref{EYM2},
\begin{align*}
\lefteqn{(-\frac{d^2}{ds^2}+\Delta_A)d_A\omega}\\
&=-\frac{d}{ds}(d_A\dot\omega+[\dot A\wedge\omega])+d_Ad_A^{\ast}d_A\omega+d_A^{\ast}d_Ad_A\omega\\
&=-d_A\ddot\omega-[\dot A\wedge\dot\omega]-[\ddot A\wedge\omega]-[\dot A\wedge\dot\omega]+d_Ad_A^{\ast}d_A\omega+d_A^{\ast}[F_A\wedge\omega]\\
&=-d_A\dot\omega+d_A\frac{d}{ds}\ast F_A-2[\dot A\wedge\dot\omega]-[\ddot A\wedge\omega]+d_Ad_A^{\ast}d_A\omega+d_A^{\ast}[F_A\wedge\omega]\\
&=-d_A\dot\omega+d_A\ast d_A\dot A-2[\dot A\wedge\dot\omega]-[\ddot A\wedge\omega]+d_Ad_A^{\ast}d_A\omega+d_A^{\ast}[F_A\wedge\omega]\\
&=-d_A\dot\omega-2[\dot A\wedge\dot\omega]-[\ddot A\wedge\omega]+d_A^{\ast}[F_A\wedge\omega]\\
&=-d_A\dot\omega-2[\dot A\wedge\dot\omega]-[(\ast d_A\omega+\ast[\dot A\wedge\omega])\wedge\omega]-[\ast F_A\wedge\ast d_A\omega].
\end{align*}
By \eqref{EYM2} again, term $\mathord{\mathrm{II}}$ reduces to 
\begin{align*}
(-\frac{d^2}{ds^2}+\Delta_A)(-\omega+\ast F_A)&=\ddot\omega-\ast\frac{d}{ds}d_A\dot A-\Delta_A\omega+d_A^{\ast}d_A\ast F_A\\
&=\ddot\omega-\ast d_A\ddot A-\ast[\dot A\wedge\dot A]-\Delta_A\omega+d_A^{\ast}d_A\ast F_A\\
&=\Delta_A\omega+\dot\omega+d_A^{\ast}[\dot A\wedge\omega]-\ast[\dot A\wedge\dot A].
\end{align*}
We then estimate  (for any $\eps>0$ and some absolute constant $C>0$), using H\"older's inequality several times
\begin{align*}
|\mathord{\mathrm{I}}|\leq&|\langle d_A\omega,\{F_A,d_A\omega\}+\{R_{\Sigma},d_A\omega\}\rangle|\\
&+|\langle d_A\omega,-d_A\dot\omega-2[\dot A\wedge\dot\omega]-[(\ast d_A\omega+\ast[\dot A\wedge\omega])\wedge\omega]-[\ast F_A\wedge\ast d_A\omega]\rangle|\\
\leq&C(|d_A\omega|^3+|F_A|^3+\eps|d_A\dot\omega|^2+\eps^{-1}|d_A\omega|^2+|\dot A|^3+|\dot\omega|^3+|\omega|^3+|\omega|^6)
\end{align*}
and
\begin{align*}
|\mathord{\mathrm{II}}|\leq&|\langle-\omega+\ast F_A,\{F_A,-\omega+\ast F_A\}+\{R_{\Sigma},-\omega+\ast F_A\}\rangle|\\
&+|\langle-\omega+\ast F_A,\Delta_A\omega+\dot\omega+d_A^{\ast}[\dot A\wedge\omega]-\ast[\dot A\wedge\dot A]\rangle|\\
\leq&C(|-\omega+\ast F_A|^3+|F_A|^3+\eps|\nabla_Ad_A\omega|^2+\eps^{-1}|-\omega+\ast F_A|^2+|\dot\omega|^2\\
&+\eps|\nabla_A\dot A|^2+\eps^{-1}|-\omega+\ast F_A|^3+\eps^{-1}|\omega|^6+|\nabla_A\omega|^3+|\dot A|^3).
\end{align*}
Now we fix $\eps< C^{-1}$. The asserted inequality then follows. The case of nonvanishing perturbations $(X_f,Y_f)$ follows similarly, the extra terms coming from $(X_f,Y_f)$ being controlled as in \cite[Proposition D.1]{SalWeh}.   This completes the proof.
\end{proof}

We recall from \cite{Wehrheim1} the following mean value inequality.

\begin{thm}[Elliptic mean value inequality]\label{Ellipticmeanvalue}
For every $n\in\N$ there exists constants $C,\mu>0$, and $\delta>0$ such that the following holds for all metrics $g$ on $\R^n$ such that $\|g-\mathbbm 1\|_{W^{1,\infty}}<\delta$. Let $B_r(0)\subseteq\R^n$ be the geodesic ball of radius $0<r\leq1$. Suppose that the nonnegative function $e\in C^2(B_r(0),[0,\infty))$ satisfies for  some $A_0,A_1,a\geq0$
\begin{eqnarray*}
\Delta e\leq A_0+A_1e+a e^{(n+2)/n}\qquad\textrm{and}\qquad\int_{B_r(0)}e\leq\mu a^{-n/2}.
\end{eqnarray*}
Then
\begin{eqnarray*}
e(0)\leq CA_0r^2+C\big(A_1^{n/2}+r^{-n}\big)\int_{B_r(0)}e.
\end{eqnarray*}
\end{thm}

\begin{proof}
For a proof we refer to \cite[Theorem 1.1]{Wehrheim1}.
\end{proof}

This result together with Proposition \ref{prop:energyineq} implies the following $L^{\infty}$ bound which we made use of in  \textsection \ref{sect:exponentialdecay}.

\begin{lem}\label{lem:pointwiseenergydesity}
Assume that $(A,\omega,\Psi)$ is a solution of \eqref{pertEYM}. Then the energy density $e=e(A,\omega,\Psi)$  satisfies on $  I\times \Sigma$ the estimate
\begin{eqnarray}\label{eq:estenergyineq}
\|e\|_{L^{\infty}(I\times\Sigma)}\leq C\|e\|_{L^1(I\times\Sigma)}
\end{eqnarray}
for a positive constant $C$ which only depends on $|I|$ and on $\Sigma$.  
\end{lem}

\begin{proof}
Theorem \ref{Ellipticmeanvalue} along with Proposition  \ref{prop:energyineq} yields for suitable constants $C_0$ and $C_1$ and every $(s,z)\in I\times\Sigma$ the estimate
\begin{eqnarray*}
e(s,z)\leq C_0(\|\omega\|_{L^{\infty}(I\times\Sigma)}+\|A\|_{L^{\infty}(I\times\Sigma)})+C_1\int_{I\times\Sigma}e.
\end{eqnarray*}
Taking the supremum over $I\times\Sigma$ the claim follows, since we may absorb the second term using for $\eps>0$ sufficiently small the estimate
\begin{eqnarray*}
\|\omega\|_{L^{\infty}(I\times\Sigma)}\leq\varepsilon\|\nabla_A\omega\|_{L^{\infty}(I\times\Sigma)}+\varepsilon^{-1}\|\omega\|_{L^2(I\times\Sigma)},
\end{eqnarray*}
and likewise for $\|A\|_{L^{\infty}(I\times\Sigma)}$.
\end{proof}

\section{Further auxiliary results}

By gauge-invariance of the subsequent estimates we may in the following assume that $\Phi=0$. Let $(A,\omega)$ be a solution of \eqref{pertEYM1} such that for critical points $(A^{\pm},\omega^{\pm})$ the asymptotic condition
\begin{eqnarray*}
\lim_{s\to\pm\infty}(A(s),\omega(s))=(A^{\pm},\omega^{\pm})
\end{eqnarray*}
is satisfied. Let $C^{\pm}:=(\J+h_f)(A^{\pm},\omega^{\pm})$. Since $(A,\omega)$ is an $L^2$-gradient flow line of $\J+h_f$ it satisfies the energy identity
\begin{eqnarray}\label{eq:energyidentity}
\int_{-\infty}^{\infty}\|\nabla(\J+h_f)(A(s),\omega(s))\|_{L^2(\Sigma)}^2\,ds=C^--C^+\geq0.
\end{eqnarray}
From this we obtain the following estimate for the $L^2$-norm of $F_A$.

\begin{lem}\label{lem:L2boundFA} 
For $(A,\omega)$ as above and any interval $I\subseteq\R$ there holds the estimate
\begin{eqnarray*}
\frac{1}{2}\|F_A\|_{L^2(I\times\Sigma)}^2\leq(1+|I|)C^--C^++|I|C_f  
\end{eqnarray*}
for a positive constant $C_f$ which only depends on the perturbation $h_f$.
\end{lem}

\begin{proof}
Since for all $(A,\omega)$, $\|\nabla(\J+h_f)(A,\omega)\|_{L^2(\Sigma)}^2\leq\|\omega-\ast F_A-Y_f\|_{L^2(\Sigma)}^2$ it follows from the energy identity \eqref{eq:energyidentity} that
\begin{eqnarray*}
C^--C^+&\geq&\int_I\|\omega-\ast F_A-Y_f\|_{L^2(\Sigma)}^2\,ds\\
&\geq&\frac{1}{2}\int_I\|\omega-\ast F_A\|_{L^2(\Sigma)}^2-\|Y_f\|_{L^2(\Sigma)}^2\,ds\\
&=&\frac{1}{2}\int_I\|F_A\|_{L^2(\Sigma)}^2-2\langle\omega,\ast F_A\rangle+\|\omega\|_{L^2(\Sigma)}^2-\|Y_f\|_{L^2(\Sigma)}^2\,ds\\
&=&\frac{1}{2}\int_I\|F_A\|_{L^2(\Sigma)}^2-\J(A,\omega)-\|Y_f\|_{L^2(\Sigma)}^2\,ds\\
&\geq&\frac{1}{2}\|F_A\|_{L^2(I\times\Sigma)}^2-|I|C^--|I|C_f.
\end{eqnarray*}
For the last estimate we used that $\J(A(s),\omega(s))\leq C^-$ for all $s\in I$ and further that $\|Y_f\|_{L^2(\Sigma)}^2\leq C_f$ for a constant $C_f$ which does not depend on $(A,\omega)$ (cf.~\cite[Proposition D.1 (iii)]{SalWeh}). The claim hence follows.
\end{proof}

An estimate for the $L^2$-norm of $\omega$ is obtained in the following lemma.

\begin{lem}\label{lem:L2estomega}
For $(A,\omega)$ as above and every $s\in\R$ there holds the estimate
\begin{eqnarray*}
\frac{1}{2}\|\omega(s)\|_{L^2(\Sigma)}\leq\|F_{A(s)}\|_{L^2(\Sigma)}+C_f.
\end{eqnarray*}
for a positive constant $C_f$ which only depends on the perturbation $h_f$.
\end{lem}

\begin{proof}
Note that $\omega^+=F_{A^+}$ because $(A^+,\omega^+)\in\crit(\J+h_f)$, and therefore 
\begin{eqnarray*}
C^+=(\J+h_f)(A^+,\omega^+)=\frac{1}{2}\int_{\Sigma}|F_{A^+}|^2\,\dvol_{\Sigma}+h_f(A^+)\geq h_f(A^+).
\end{eqnarray*}
It therefore follows from the gradient flow property and the Cauchy--Schwartz inequality that for all $s\in\R$
\begin{eqnarray*}
h_f(A^+)&\leq&(\J+h_f)(A(s),\omega(s))\\
&=&\int_{\Sigma}\langle F_{A(s)},\omega(s)\rangle-\frac{1}{2}|\omega(s)|^2\,\dvol_{\Sigma}+h_f(A(s),\omega(s))\\
&\leq&\|F_{A(s)}\|_{L^2(\Sigma)}\|\omega(s)\|_{L^2(\Sigma)}-\frac{1}{2}\|\omega(s)\|_{L^2(\Sigma)}^2+h_f(A(s),\omega(s)).
\end{eqnarray*}
By definition of $h_f$ there exists a constant $C_f$ such that $|h_f(A,\omega)|\leq C_f$ for all $(A,\omega)$. Thus we can further estimate
\begin{eqnarray*}
\frac{1}{2}\|\omega(s)\|_{L^2(\Sigma)}^2-\|F_{A(s)}\|_{L^2(\Sigma)}\|\omega(s)\|_{L^2(\Sigma)}-2C_f\leq0,
\end{eqnarray*}
which implies the result.
\end{proof}

\end{document}